\documentclass{article}
\usepackage{amsmath}
\usepackage{amsfonts}
\usepackage{amssymb}
\usepackage{tikz-cd}
\usepackage{graphicx}

\newtheorem{theorem}{Theorem}
\newtheorem{definition}[theorem]{Definition}
\newtheorem{lemma}[theorem]{Lemma}
\newtheorem{proposition}[theorem]{Proposition}
\newtheorem{remark}[theorem]{Remark}
\newenvironment{proof}[1][Proof]{\noindent\textbf{#1.} }{\ \rule{0.5em}{0.5em}}
\begin{document}

\title{Analytic one-dimensional maps and two-dimensional ordinary differential
	equations can robustly simulate Turing machines}
\author{Daniel S. Gra\c{c}a\\Universidade do Algarve, C. Gambelas, 8005-139 Faro, Portugal\\\& Instituto de Telecomunica\c{c}\~{o}es, Lisbon, Portugal
	\and N. Zhong\\DMS, University of Cincinnati, Cincinnati, OH 45221-0025, U.S.A.}
\maketitle

\begin{abstract}
	In this paper, we analyze the problem of finding the minimum
	dimension $n$ such that a closed-form analytic map/ordinary
	differential equation can simulate a Turing machine over
	$\mathbb{R}^{n}$ in a way that is robust to perturbations. We show
	that one-dimensional closed-form analytic maps are sufficient to
	robustly simulate Turing machines; but the minimum dimension for the
	closed-form analytic ordinary differential equations to robustly
	simulate Turing machines is two, under some reasonable
	assumptions. We also show that any Turing machine can be simulated by a two-dimensional $C^{\infty}$
	ordinary differential equation on the
	compact sphere $\mathbb{S}^{2}$.
\end{abstract}

\section{Introduction}

As it is well known (see e.g.\ \cite{Sip12}), a Turing machine is a
mathematical model of computation that formalizes the notion of
algorithm/computation over discrete structures such as the set of
positive integers $\mathbb{N}$ or integers $\mathbb{Z}$. In
practice, Turing machines are computationally equivalent to a
standard digital computer. Furthermore, following the work of Turing
and others it is also known that some problems such as the Halting
Problem or Hilbert's 10th problem are \emph{noncomputable}, i.e.\
there is no Turing machine (i.e.\ no algorithm) that solves those
problems \cite{Tur36}, \cite{Mat93}. This remarkable result shows
that some problems are algorithmically unsolvable. Examples of other
nontrivial behavior regarding Turing machines include, for instance,
the existence of \emph{universal Turing machines}, which can
simulate the computation of any other Turing machine (see
\cite{Tur36} or e.g.\ \cite{Sip12}), or the existence of
self-reproducing Turing machines which output their own description
\cite{Mos10}.

Although the results above are considered classically over discrete
structures (e.g.\ $\mathbb{N}$), often they can be studied over
continuous spaces such as $\mathbb{R}^{n}$. The idea is to simulate
the computation of a Turing machine with a continuous map/flow. If a
continuous system is able to simulate any Turing machine (or,
equivalently, a universal Turing machine), then this system is
usually referred to as \emph{Turing universal}. A consequence of the
noncomputability of the Halting Problem is that the long term
behavior of Turing universal systems is highly complex (in a manner
distinct from behaviors considered e.g.\ in chaos theory) and has
some characteristics which are not computable (see e.g.\
\cite{Moo91}). However, in applications, it is often desirable that
Turing universal systems are relatively simple and mathematically
well-behaved so that they can be used in meaningful situations. For
this reason, one might be interested in having properties such as
low-dimensionality, reasonable smoothness, or robustness to
perturbations for Turing universal systems.

In this paper, we investigate the problem of determining the lowest
dimension $n$ such that the analytic maps/ODEs defined on
$\mathbb{R}^n$ can robustly simulate Turing machines.

It is well-known that piecewise affine and other types of maps and
ODEs can simulate Turing machines on $\mathbb{R}^{n}$, see
e.g.~\cite{Moo91}, \cite{KCG94}, \cite{Bra95}, \cite{Koi96},
\cite{Bou99}, \cite{KP05}, \cite{BC08}, \cite{BP18}). However, some
authors have claimed (see e.g.~\cite{MO97}, \cite{KM99},
\cite{AB01}, \cite{KM99}, \cite{BGH13}) that, when focusing on more
physically realistic systems simulating Turing machines, one might
desire other additional attributes such as robustness to noise (it
is known that the addition of noise to some classes of systems
reduces their computational power, see e.g.~\cite{MO97},
\cite{AB01}) or smoothness of the dynamics since most classical
physical systems are expressed with smooth (actually analytic)
functions. In \cite{KM99} the authors have shown that closed-form
analytic maps (i.e.~those analytic functions which can be expressed
in terms of elementary functions such as
polynomials, trigonometric functions, exponential and logarithmic
functions, their composition, etc.) are capable of simulating Turing
machines with exponential slowdown in dimension one or in real time
in dimension $\geq 2$. In \cite{GCB08} it was shown that, under a
certain notion of robustness (see Theorems \ref{Th:Main} and
\ref{Th:Simulation} below), the class of closed-form analytic maps on
$\mathbb{R}^{3}$ as well as the class of ODEs defined with analytic
closed-form functions in $\mathbb{R}^{6}$ can robustly simulate
Turing machines.

In the present paper, we show that one-dimensional closed-form
analytic maps can robustly simulate Turing machine on $\mathbb{R}$
(Theorem \ref{Th:one-dim-Turing-map}) in real time (i.e.~without the
exponential slowdown of \cite{KM99}; the simulation in \cite{KM99}
is not robust to perturbations either) and that two-dimensional
closed-form analytic ODEs can also robustly simulate Turing machines
(Theorem \ref{Th:two-dim-ODE-Turing}), both in the sense of
\cite{GCB08}. We also show that, under certain reasonable
assumptions, none of one-dimensional autonomous analytic ODEs can
simulate (robustly or not) a universal Turing machine.

Similar to what is done in \cite{CMP21}, we show that there is a
$C^{\infty
}$ ODE which can simulate Turing machines over the compact set $\mathbb{S}
^{n}=\{x\in\mathbb{R}^{n+1}:\left\Vert x\right\Vert =1\}$ for
$n\geq2$. The difference between our result and the result of
\cite{CMP21} is that, although the flow of \cite{CMP21} is
mathematically simpler (it is polynomial), Turing universality is
only achieved over $\mathbb{S}^{n}$ for $n\geq17$. The authors then
use this polynomial flow to construct a (Euler) partial differential
equation which is Turing universal. We also note that in
\cite{CMPP21} the authors proved that there are Turing complete
(stationary Euler) fluid-flows on a Riemannian 3-dimensional sphere.
The difference of the later result from the one presented in this
paper is that we use ODEs instead of partial differential equations,
and our results are for $\mathbb{S}^{2}$ instead of
$\mathbb{S}^{3}$.

The outline of the paper is as follows. In Section
\ref{Sec:Dimension3} we review the construction presented in
\cite{GCB08} to create analytic maps on $\mathbb{R}^{3}$ which can
robustly simulate Turing machines. In Section \ref{Sec:auxiliary} we
present some auxiliary functions. Building on these results, we show
in Section \ref{Sec:Turing-map} that one-dimensional analytic maps
can robustly simulate Turing machines. By iterating these maps with
ODEs, we are able to show in Section \ref{Sec:ODE-simulation} that
two-dimensional ODEs can robustly simulate Turing machines. In
Section \ref{Sec:Compact}, we construct a $C^{\infty}$ ODE that can
simulate Turing machines over the compact set $\mathbb{S}^{2}$.
Finally, in Section \ref{Sec:ODE-one-dimension} we show that under
reasonable hypothesis, no one-dimensional analytic ODE can simulate
a universal Turing machine.

\section{Simulating Turing machines in dimension three\label{Sec:Dimension3}}

In this section, we review several results from \cite{GCB08} which
are useful for proving our main results.

We first recall some basic results from computability theory (see
e.g. \cite{Sip12}). Given a finite set $\Sigma$ (the
\emph{alphabet}), a \emph{word} over $\Sigma$ is a finite sequence
$w=(w_{1},...,w_{k})\in \Sigma^{k}$ for some $k\in\mathbb{N}_{0}$
($k$ is the length of the word), where $\mathbb{N}_0$ is the set of
all non-negative integers. Note that there is a special sequence,
represented by $\epsilon$, which denotes the word of length $0$. As
usual, for notational simplicity, we will denote the word
$w=(w_{1},...,w_{k})$ simply as $w=w_{1}...w_{k}$. The set of all
words over $\Sigma$ is denoted by $\Sigma^{\ast}$. We also recall
that a Turing machine is a discrete dynamical system defined by the
iteration of a map, although it is usually viewed as a finite-state
machine since this approach is often more convenient. More
specifically, let $\Sigma$ be an alphabet, and take some symbol
$B\notin\Sigma$, which is usually known as the \textit{blank
symbol}, and let $Q$ be a finite set known as the set of
\textit{states} with some special elements $q_{0},q_{h}\in Q$,
called the initial state and the final state, respectively. Then a
Turing machine $M$ is a map
$F_{M}:\Sigma^{\ast}\times\Sigma^{\ast}\times
Q\rightarrow\Sigma^{\ast }\times\Sigma^{\ast}\times Q$ that works as
follows when viewed as a machine. It has a bi-infinite tape, divided
into \emph{cells}, and a head which is
associated to some state of $Q$. Given some $(u,v,q)\in\Sigma^{\ast}
\times\Sigma^{\ast}\times Q$ (the \emph{configuration} of the Turing machine),
where $u=u_{1}u_{2}\ldots u_{n}$ and $v=v_{1}v_{2}\ldots v_{p}$, then the tape
contents of the Turing machine at this configuration is
\begin{equation}
...B\,B\,B\,v_{p}...\,v_{2}v_{1}\,u_{1}u_{2}\ldots u_{n}\,B\,B\,B...,
\label{Eq:config}
\end{equation}
while its associated state is $q$. In this case the Turing machine is also
said to be reading symbol $v_{1}$. Then, \textit{depending only on the value
of }the current state and of the symbol being read by the head, the machine
simultaneously (i) updates its state, (ii) updates the symbol being read by
the head and (iii) either moves the head one cell to the right, one cell to
the left, or maintains the head on the same position.

A Turing machine $M$ computes a function $f:\Sigma^{\ast}\rightarrow
\Sigma^{\ast}$ as follows. Given a word $w$ it starts its
computation on the initial configuration $(w,\epsilon,q_{0})$, i.e.
in the initial configuration the state is the initial state and the
tape contains the input $w$ only. Then $M$ proceeds with the
computation until it reaches the halting state $q_{h}$. In this case
we say that the Turing machine has \emph{halted} with configuration
$(u_{h},v_{h},q_{h})\in\Sigma^{\ast}\times\Sigma^{\ast}\times Q$. In
this case its output will be $u_{h}$, i.e. $u_{h}=f(w)$. If $M$ does
not halt with input $w$, then $f(w)$ is undefined.

Given some Turing machine $M$ as described above, let $k=1+\#\Sigma$ and take
an injective map $\gamma:\Sigma\rightarrow\{0,1,2,\ldots,k-1\}$ with
$\gamma(B)=0$. Let $(u,v,q)$ be the current configuration of $M$ and let us
further assume that $M$ has $m$ states, represented by the numbers
$1,\ldots,m$, and that if $M$ reaches an halting configuration, then it moves
to the same configuration (i.e.$\ F_{M}(u_{h},v_{h},q_{h})=(u_{h},v_{h}
,q_{h})$). Take
\begin{equation}
\begin{array}
[c]{l}
y_{1}=\gamma(u_{1})+\gamma(u_{2})k+\ldots+\gamma(u_{n})k^{n-1}\\
y_{2}=\gamma(v_{1})+\gamma(v_{2})k+\ldots+\gamma(v_{p})k^{p-1}
\end{array}
\label{Eq:ConfigEncoding}
\end{equation}
and suppose that $q$ is the state associated to the current
configuration. Then $(y_{1},y_{2},q)\in\mathbb{N}^{3}$ encodes
unambiguously the current configuration of $M$. Under these
assumptions, the transition function of $M$ can be encoded as a
function $\bar{f}_{M}:\mathbb{N}^{3}\rightarrow \mathbb{N}^{3}$. In
\cite{GCB08} it was shown that $\bar{f}_{M}$ can be extended to a
function $f_{M}:\mathbb{R}^{3}\rightarrow\mathbb{R}^{3}$, which has
the following properties: (i) it is capable of simulating $M$ in the
presence of perturbations; (ii) the function $f$ is analytic, and
each of its components can be expressed using only the following
terms: variables, polynomial-time computable constants (see Remark
\ref{Remark:computable} for a definition), $+$, $-$, $\times$,
$\sin$, $\cos$, $\arctan$. The precise statement of this result is
given below, where $\left\Vert f\right\Vert =\sup_{x\in A}\left\Vert
f(x)\right\Vert $ for a function $f:A\subseteq
\mathbb{R}^{l}\rightarrow\mathbb{R}^{j}$, $\left\Vert y\right\Vert
=\max_{1\leq i\leq j}\left\vert y_{i}\right\vert $ for
$y=(y_{1},\ldots ,y_{j})\in\mathbb{R}^{j}$, and $f^{[k]}$ denotes
the $k$th iterate of the function $f:A\rightarrow A$, which is
defined as follows: $f^{[0]}(x)=x$, $f^{[k+1]}(x)=f^{[k]}(f(x))$.

\begin{theorem}
[{\cite[p. 333]{GCB08}}]\label{Th:Main}Let $\psi:\mathbb{N}^{3}\rightarrow
\mathbb{N}^{3}$ be the transition function of a Turing machine $M$ under the
encoding described above, and let $0<\delta<\varepsilon<1/2$. Then $\psi$
admits a globally analytic closed-form extension $f_{M}:\mathbb{R}
^{3}\rightarrow\mathbb{R}^{3}$ such that the expression of each component of
$f_{M}$ can be written using only the following terms: variables,
polynomial-time computable constants, $+$, $-$, $\times$, $\sin$, $\cos$,
$\arctan$. Moreover, $f_{M}$ is robust to perturbations in the following
sense: for all $f$ such that $\left\Vert f-f_{M}\right\Vert \leq\delta$, for
all $j\in\mathbb{N}$, and for all $\bar{x}_{0}\in\mathbb{R}^{3}$ satisfying
$\left\Vert \bar{x}_{0}-x_{0}\right\Vert \leq\varepsilon,$ where $x_{0}
\in\mathbb{N}^{3}$ represents a configuration according to the encoding
described above,
\[
\left\Vert f^{[j]}(\bar{x}_{0})-\psi^{\lbrack j]}(x_{0})\right\Vert
\leq\varepsilon.
\]

\end{theorem}

We note that the proof of this theorem is constructive and that
$f_M$ can be obtained explicitly. A continuous-time version of
Theorem \ref{Th:Main} was also proved in \cite{GCB08}.

\begin{remark}
\label{Remark:computable}We note that we can define computable real
constants and computable real functions using the approach of
computable analysis. These notions can be presented in several
equivalent but different ways. For example, according to the
approach presented in \cite{Ko91} (see also \cite{BHW08}), a number
$c\in\mathbb{R}$ is computable if there is a Turing machine $M$
that, on input $n\in\mathbb{N}$, outputs (in finite time) a rational
$q_{n}$ with the property that $\left\vert q_{n}-c\right\vert
\leq2^{-n}$. If the Turing machine $M$ runs in polynomial time (in
$n$), then we say that $c$ is computable in polynomial time.
Similarly, a function $f:\mathbb{R\rightarrow R}$ is computable if
there is an oracle Turing machine $M$ that computes $f(x)$ in the
sense that, given as input $n\in \mathbb{N}$ and any oracle
$\varphi: \mathbb{N}\to \mathbb{Q}$ recording $x\in \mathbb{R}$
(i.e.~with the property that $\left\vert x-\varphi(n)\right\vert
\leq2^{-n}$), $M$ outputs a rational number $q_{n}$ such that
$\left\vert q_{n}-c\right\vert \leq2^{-n}$.
A $C^{1}$ real function
$f:\mathbb{R\rightarrow R}$ is $C^{1}$-computable if $f$ and its
derivative are both computable. These notions can be generalized to
$\mathbb{R}^{n}$ in a straightforward manner.
\end{remark}

\begin{theorem}
[{\cite[p. 333]{GCB08}}]\label{Th:Simulation}Let $\psi:\mathbb{N}
^{3}\rightarrow\mathbb{N}^{3}$ be the transition function of a
Turing machine $M$ under the encoding described above; let
$0<\varepsilon\leq1/4$; and let $0\leq\delta<2\varepsilon\leq1/2$.
Then there exist

\begin{itemize}
\item $\eta>0$ satisfying $\eta<1/2$, which can be computed from
$\psi,\varepsilon,\delta$, and

\item an analytic closed-form function $g_{M}:\mathbb{R}^{7}\rightarrow
\mathbb{R}^{6}$ which can be written using only the following terms:
variables, polynomial-time computable constants, $+$, $-$, $\times$, $\sin$,
$\cos$, $\arctan$
\end{itemize}

\noindent such that the ODE $z^{\prime}=g_{M}(t,z)$ robustly simulates $M$ in
the following sense: for all $g$ satisfying $\left\Vert g-g_{M}\right\Vert
\leq\delta<1/2$ and for every $x_{0}\in\mathbb{N}^{3}$ that encodes a
configuration according to the encoding described above, if $\bar{x}_{0}
,\bar{y}_{0}\in\mathbb{R}^{3}$ satisfy the conditions $\left\Vert \bar{x}
_{0}-x_{0}\right\Vert \leq\varepsilon$ and $\left\Vert \bar{y}_{0}
-x_{0}\right\Vert \leq\varepsilon$, then the solution $z(t)$ of
\[
z^{\prime}=g(t,z),\qquad z(0)=(\bar{x}_{0},\bar{y}_{0})
\]
satisfies, for all $j\in\mathbb{N}$ and for all $t\in\lbrack j,j+1/2]$,
\begin{equation}
\left\Vert z_{2}(t)-\psi^{\lbrack j]}(x_{0})\right\Vert \leq\eta,
\label{Eq:ODE_bound}
\end{equation}
where $z\equiv(z_{1},z_{2})$, with $z_{1}\in\mathbb{R}^{3}$ and $z_{2}
\in\mathbb{R}^{3}$.
\end{theorem}

\section{Some useful auxiliary functions\label{Sec:auxiliary}}

In this section we present several functions and results which are
needed in subsequent sections. The function $\Upsilon$ presented in
the next lemma can be seen as a generalization of the
error-correcting function $l_{2}$ from \cite[Lemma 9]{GCB08}. More
specifically, the function $\Psi$ can be viewed as a function that
improves the accuracy of approximations within distance $\leq1/5$ of
an integer (the function $l_{2}$ of \cite[Lemma 9]{GCB08} has a
similar property, but only works for the integers $0$ and 1$)$,
where the correction factor is bounded by $e^{-y}$, and $y>0$ is the
second argument of $\Psi$.

\begin{lemma}
\label{Lemma:Psi}Let $\Psi:\mathbb{R}^{2}\rightarrow\mathbb{R}$ be given by
$\Psi(x,y)=x-\frac{1}{2\pi}\arcsin(\sin(2\pi x)(1-e^{-y-2})).$ Then
$\left\vert \Psi(x,y)-k\right\vert <e^{-y}\left\vert x-k\right\vert $ whenever
$\left\vert x-k\right\vert \leq1/5$ for some $k\in\mathbb{Z}$ and $y\geq0$.
\end{lemma}

\begin{proof}
Let $\overline{\Psi}:\mathbb{R}\rightarrow\mathbb{R}$ be defined by
$\overline{\Psi}(x)=x-\frac{1}{2\pi}\arcsin(\sin(2\pi x))$. We note that,
since $\sin(2\pi x)$ has period 1, $\overline{\Psi}(x)=k$ if $x\in\lbrack
k-1/4,k+1/4]$ for some $k\in\mathbb{Z}$. However, although it is continuous,
the function $\overline{\Psi}$ is not analytic, since it is well-known that if
an analytic function is constant in a non-empty interval, e.g.~$[3/4,5/4]$,
then it should be constant everywhere on the real line $\mathbb{R}$, which is
not the case for $\overline{\Psi}$. The problem is that, although the
composition of analytic functions yields again an analytic function, the
derivative of $\arcsin y$ is not defined when $y=-1$ or $y=1$ and thus
$\overline{\Psi}$ is not analytic when $x=k-1/4$ or $x=k+1/4$ for some
$k\in\mathbb{Z}$. Note, however, that $\arcsin$ is analytic in $(-1,1)$.
Hence, we can multiply $\sin(2\pi x)$ by a value $1-e^{-y}$ (or $1-e^{-y-2}$,
which will be more convenient later on), which is slightly less than 1, to
ensure that the resulting function $\Psi$ is analytic, since in this way we
guarantee that $-1<\sin(2\pi x)(1-e^{-y})<1$ for any $x\in\mathbb{R}$ and
$y\geq0$. Next we notice that, by the mean value theorem
\begin{align*}
\left\vert \arcsin a-\arcsin b\right\vert  &  \leq\left\vert a-b\right\vert
\max_{x\in\lbrack a,b]}\frac{1}{\sqrt{1-x^{2}}}\\
&  =\left\vert a-b\right\vert \max\left(  \frac{1}{\sqrt{1-a^{2}}},\frac
{1}{\sqrt{1-b^{2}}}\right)  .
\end{align*}
Let us now take $g(x)=7x-\sin(2\pi x)$. We note that $g(0)=0$ and that
$g^{\prime}(x)=7-2\pi\cos(2\pi x)>0$. Hence we conclude that $g$ strictly
increases in $[0,1/5]$, which implies $7\left\vert x\right\vert \geq\left\vert
\sin(2\pi x)\right\vert $ when $x\in\lbrack-1/5,1/5]$. This implies that for
$x\in\lbrack k-1/5,k+1/5]$, where $k\in\mathbb{Z}$ is arbitrary, we have
\begin{align*}
\left\vert \Psi(x,y)-k\right\vert  &  =\left\vert x-\frac{1}{2\pi}\arcsin
(\sin(2\pi x)(1-e^{-y-2}))-\left(  x-\frac{1}{2\pi}\arcsin(\sin(2\pi
x))\right)  \right\vert \\
&  =\frac{1}{2\pi}\left\vert \arcsin(\sin(2\pi x))-\arcsin(\sin(2\pi
x)(1-e^{-y-2}))\right\vert \\
&  \leq\frac{1}{2\pi}\frac{1}{\sqrt{1-\sin^{2}(2\pi/5)}}\left\vert \sin\left(
2\pi x\right)  (1-(1-e^{-y-2}))\right\vert \\
&  =\frac{1}{2\pi}\frac{1}{\sqrt{1-\sin^{2}(2\pi/5)}}\left\vert \sin\left(
2\pi(x-k)\right)  \right\vert e^{-y-2}\\
&  <\left\vert \sin\left(  2\pi(x-k)\right)  \right\vert e^{-y}e^{-2}\\
&  \leq7\left\vert x-k\right\vert e^{-y}e^{-2}\\
&  \leq\left\vert x-k\right\vert e^{-y}.
\end{align*}
\end{proof}

We now present another error-correcting function $\sigma:\mathbb{R\rightarrow
R}$ which was first presented in \cite[Proposition 5]{GCB08}. This function is
a uniform contraction around integers. Unlike $\Psi$, one cannot prescribe the
amount of error reduction around each integer with a single application of the
map $\sigma$. On the other hand its use is not restricted to a $1/5$
-neighborhood of integers and can be used on larger neighborhoods. This last
property will be handy later on.

\begin{lemma}
[\cite{GCB08}]\label{Lemma:sigma}Let $\sigma:\mathbb{R\rightarrow R}$ be the
function defined by $\sigma(x)=x-0.2\sin(2\pi x)$. Let $\varepsilon\in
\lbrack0,1/2).$ Then there is some contracting factor $\lambda_{\varepsilon
}\in(0,1)$ such that, $\forall\delta\in\lbrack-\varepsilon,\varepsilon],$
$\forall n\in\mathbb{Z}$, $|\sigma(n+\delta)-n|<\lambda_{\varepsilon}\delta.$
\end{lemma}

The constants $\lambda_{\varepsilon}$ can usually be explicitly obtained. For
example, as shown in \cite{GCB08}, we can take $\lambda_{1/4}=0.4\pi
-1\approx0.2566371$.

It is well known that there are bijective functions from $\mathbb{N}^{2}$ to
$\mathbb{N}$. An example (see e.g.~\cite[pp. 26--27]{Odi89}) is the
dovetailing pairing map $I:\mathbb{N}^{2}\mathbb{\rightarrow N}$ defined by
the formula
\begin{equation}
I(x,y)=\frac{(x+y)^{2}+3x+y}{2}. \label{Eq:bijection-dim2-dim1}
\end{equation}
Using this map we can obtain a bijective map $I_{k}:\mathbb{N}^{k}
\rightarrow\mathbb{N}$, for $k\geq2$, by defining $I_{k}$ recursively:
$I_{2}(x_{1},x_{2})=I(x_{1},x_{2})$; $I_{k+1}(x_{1},\ldots,x_{k}
,x_{k+1})=I_{2}(I_{k}(x_{1},\ldots,x_{k}),x_{k+1})$. We now show that the maps
$I_{k}$ can be extended to $\mathbb{R}^{k}$ robustly around the integers.
Since each $I_{k}$ is a (multivariate) polynomial, to achieve this objective
we have to analyze how the error is propagated via the application of a
polynomial map. The following lemma is from \cite{BGP12}, and can be viewed as
an extension of a similar result proved in \cite[Lemma 11]{GCB08}\ for the
case of monomials. For multivariate polynomials, the multi-index notation is
used for compactness as follows: a monomial $x_{1}^{\alpha_{1}}\ldots
x_{k}^{\alpha_{k}}$ is represented by $x^{\alpha}$, where $x=(x_{1}
,\ldots,x_{k})$, $\alpha=(\alpha_{1},\ldots,\alpha_{k})$, $\left\vert
\alpha\right\vert =\alpha_{1}+\ldots+\alpha_{k}$ is the degree of the
monomial, and the degree of a (multivariate) polynomial is the maximum degree
of all the monomials which appear in its expression.

\begin{lemma}
[{\cite[Lemma 4]{BGP12}}]\label{Lem:Modulus_continuity}Let $P:\mathbb{R}
^{k}\rightarrow\mathbb{R}$ be a multivariate polynomial of degree $k$ and let
$x,y\in\mathbb{R}^{k}$ be such that $\left\Vert x\right\Vert ,\left\Vert
y\right\Vert \leq M$ for some $M\geq0$. Then
\[
\left\vert P(x)-P(y)\right\vert \leq kM^{k-1}\Sigma P\left\Vert
x-y\right\Vert
\]
where $\Sigma P$ denotes the sum of the absolute values of the coefficients of
$P$.
\end{lemma}

Now we are ready to state the result that shows the existence of robust
analytic extensions of $I_{k}$ for each $k$.

\begin{proposition}
\label{Prop:many-to-one}For each $k\in\mathbb{N}$, $k\geq2$, there exists an
analytic function $\Upsilon_{k}:\mathbb{R}^{k}\rightarrow\mathbb{R}$ with the
following properties:

\begin{enumerate}
\item If $x\in\mathbb{N}^{k}$, then $\Upsilon_{k}(x)=I_{k}(x)$;

\item For any $x\in\mathbb{R}^{k}$, if there is some $y\in\mathbb{N}^{k}$ such
that $\left\Vert x-y\right\Vert \leq1/5$, then $\left\vert \Upsilon
_{k}(x)-I_{k}(y)\right\vert \leq\left\Vert x-y\right\Vert \leq1/5.$
\end{enumerate}
\end{proposition}

\begin{proof}
We start with the case $k=2$. Since
\[
I_{2}(x_{1},x_{2})=\frac{(x_{1}+x_{2})^{2}+3x_{1}+x_{2}}{2}=\frac{x_{1}
^{2}+2x_{1}x_{2}+x_{2}^{2}+3x_{1}+x_{2}}{2},
\]
it is clear that the function $I_{2}$ is well-defined for all $x_{1},x_{2}
\in\mathbb{R}$. In the remaining of this proof, we assume that $I_{2}$ is
defined over $\mathbb{R}^{2}$. Since $I_{2}$ is a polynomial of degree 2, by
Lemma \ref{Lem:Modulus_continuity} we conclude that
\[
\left\vert I_{2}(x)-I_{2}(y)\right\vert \leq8\max(\left\Vert x\right\Vert
,\left\Vert y\right\Vert )\left\Vert x-y\right\Vert .
\]
Since $\left\Vert x\right\Vert \leq1+\left\Vert x\right\Vert ^{2}$ for all
$x\in\mathbb{R}^{2}$, it follows that
\begin{equation}
\left\vert I_{2}(x)-I_{2}(y)\right\vert \leq8(2+\left\Vert x\right\Vert
^{2}+\left\Vert y\right\Vert ^{2})\left\Vert x-y\right\Vert . \label{Eq:aux2}
\end{equation}
Set
\[
\Upsilon_{2}(x_{1},x_{2})=I_{2}(\Psi(x_{1},32(1+\left\Vert x\right\Vert
_{2}^{2})),\Psi(x_{2},32(1+\left\Vert x\right\Vert _{2}^{2}))),
\]
where $\left\Vert x\right\Vert _{2}^{2}=x_{1}^{2}+x_{2}^{2}$. As a composition
of analytic functions, $\Upsilon_{2}:\mathbb{R}^{2}\rightarrow\mathbb{R}$ is
clearly analytic. If $x_{1},x_{2}\in\mathbb{N}$, it is trivial to verify that
$\Upsilon_{2}(x_{1},x_{2})=I_{2}(x_{1},x_{2})$, which implies property 1. For
property 2, let us assume that $\left\Vert x-y\right\Vert \leq\varepsilon
\leq1/5$ for some $y\in\mathbb{N}^{k}$. Using Lemma \ref{Lemma:Psi} and the
inequality $e^{-l}<(1/l)$ for all $l\geq1$, we obtain the following estimate,
where $\varepsilon=\left\Vert x-y\right\Vert \leq1/5$:
\begin{align}
&  \left\Vert (\Psi(x_{1},32(1+\left\Vert x\right\Vert _{2}^{2})),\Psi
(x_{2},32(1+\left\Vert x\right\Vert _{2}^{2})))-(y_{1},y_{2})\right\Vert
\nonumber\\
&  \leq\varepsilon e^{-32(1+\left\Vert x\right\Vert _{2}^{2})}\nonumber\\
&  \leq\frac{\varepsilon}{32(1+\left\Vert x\right\Vert _{2}^{2})}.
\label{Eq:bound-I-k}
\end{align}
Recall that, on $\mathbb{R}^{2}$, $\Vert\cdot\Vert_{2}$ denotes the Euclidean
norm while $\Vert\cdot\Vert$ denotes the maximum norm. Since $\left\Vert
x-y\right\Vert \leq\frac{1}{5}$ and $\left\Vert y\right\Vert _{2}\leq\sqrt
{2}\left\Vert y\right\Vert $, it follows that $\left\Vert x-y\right\Vert
_{2}\leq\sqrt{2}\left\Vert x-y\right\Vert \leq1/2$, which further implies that
\begin{align*}
\left\Vert y\right\Vert _{2}^{2}  &  =\left\Vert y-x+x\right\Vert _{2}^{2}\\
&  \leq\left(  \left\Vert x-y\right\Vert _{2}+\left\Vert x\right\Vert
_{2}\right)  ^{2}\\
&  \leq\left(  \frac{1}{2}+\left\Vert x\right\Vert _{2}\right)  ^{2}\\
&  =\frac{1}{4}+\left\Vert x\right\Vert _{2}+\left\Vert x\right\Vert _{2}
^{2}\\
&  \leq\frac{1}{4}+1+\left\Vert x\right\Vert _{2}^{2}+\left\Vert x\right\Vert
_{2}^{2}<2+2\left\Vert x\right\Vert _{2}^{2}\text{.}
\end{align*}
Then it follows from this inequality, (\ref{Eq:aux2}), and (\ref{Eq:bound-I-k}) that
\begin{align*}
&  \left\vert \Upsilon_{2}(x)-I_{2}(y)\right\vert \\
&  =\left\vert \Upsilon_{2}(x)-\Upsilon_{2}(y)\right\vert \\
&  \leq8\max(\left\Vert x\right\Vert ,\left\Vert y\right\Vert )\left\Vert
(\Psi(x_{1},32(1+\left\Vert x\right\Vert _{2}^{2})),\Psi(x_{2},32(1+\left\Vert
x\right\Vert _{2}^{2})))-(y_{1},y_{2})\right\Vert \\
&  \leq8(2+\left\Vert x\right\Vert ^{2}+\left\Vert y\right\Vert ^{2}
)\frac{\varepsilon}{32(1+\left\Vert x\right\Vert _{2}^{2})}\\
&  \leq(2+\left\Vert x\right\Vert _{2}^{2}+\left\Vert y\right\Vert _{2}
^{2})\frac{\varepsilon}{(2+2\left\Vert x\right\Vert _{2}^{2}+(2+2\left\Vert
x\right\Vert _{2}^{2}))}\qquad(\left\Vert \cdot\right\Vert \leq\left\Vert
\cdot\right\Vert _{2})\\
&  \leq(2+\left\Vert x\right\Vert _{2}^{2}+\left\Vert y\right\Vert _{2}
^{2})\frac{\varepsilon}{(2+\left\Vert x\right\Vert _{2}^{2}+\left\Vert
y\right\Vert _{2}^{2})}=\varepsilon
\end{align*}
which proves property 2 for $k=2$.

For the case where $k>2$, the result is obtained inductively by setting
\[
\Upsilon_{k+1}(x_{1},\ldots,x_{k},x_{k+1})=\Upsilon_{2}(\Upsilon_{k}
(x_{1},\ldots,x_{k}),x_{k+1}).
\]
Property 1 is immediate; property 2 follows from the estimate below:
\begin{align*}
&  \left\Vert \Upsilon_{k+1}(x_{1},\ldots,x_{k},x_{k+1})-I_{k+1}(y_{1}
,\ldots,y_{k},y_{k+1})\right\Vert \\
&  =\left\Vert \Upsilon_{k+1}(x_{1},\ldots,x_{k},x_{k+1})-\Upsilon_{k+1}
(y_{1},\ldots,y_{k},y_{k+1})\right\Vert \\
&  =\left\Vert \Upsilon_{2}(\Upsilon_{k}(x_{1},\ldots,x_{k}),x_{k+1}
)-\Upsilon_{2}(\Upsilon_{k}(y_{1},\ldots,y_{k}),y_{k+1})\right\Vert \\
&  \leq\left\Vert (\Upsilon_{k}(x_{1},\ldots,x_{k}),x_{k+1})-(\Upsilon
_{k}(y_{1},\ldots,y_{k}),y_{k+1})\right\Vert \\
&  \leq\max\left(  \left\Vert \Upsilon_{k}(x_{1},\ldots,x_{k})-\Upsilon
_{k}(y_{1},\ldots,y_{k})\right\Vert ,\left\Vert x_{k+1}-y_{k+1}\right\Vert
\right) \\
&  \leq\max\left(  \left\Vert x_{1}-y_{1}\right\Vert ,\left\Vert x_{2}
-y_{2}\right\Vert ,\ldots,\left\Vert x_{k+1}-y_{k+1}\right\Vert \right) \\
&  \leq\left\Vert x-y\right\Vert .
\end{align*}
\end{proof}

As shown above, $I_{k}:\mathbb{N}^{k}\rightarrow\mathbb{N}$ provides a
bijection between $\mathbb{N}^{k}$ and $\mathbb{N}$ that can be robustly
extended to an analytic function $\Upsilon_{k}:\mathbb{R}^{k}\rightarrow
\mathbb{R}$. We now show that the inverse function of $I_{k}$ can also be
robustly extended to an analytic function from $\mathbb{R}$ to $\mathbb{R}
^{k}$. We write $I_{k}^{-1}(z)=(J_{k,1}(z),\ldots,J_{k,k}(z))$, where
$J_{k,1},\ldots,J_{k,k}:\mathbb{N\rightarrow N}$. The following result shows
that $I_{k}^{-1}:\mathbb{N\rightarrow N}^{k}$ can be robustly extended to an
analytic function from $\mathbb{R}$ to $\mathbb{R}^{k}$.

\begin{proposition}
\label{Prop:one-to-many}For each $k\in\mathbb{N}$, $k\geq2$, and for each
$1\leq i\leq k$, there exists an analytic function $\Omega_{k,i}
:\mathbb{R}\rightarrow\mathbb{R}$ with the following property: for any
$x\in\mathbb{R}$, if there is some $n\in\mathbb{N}$ such that $\left\vert
x-n\right\vert \leq1/5$, then $\left\vert \Omega_{k,i}(x)-J_{k,i}
(n)\right\vert \leq1/5.$
\end{proposition}

\begin{proof}
First we prove the result when $k=2$. Let us assume that $i=1$ (the case where
$i=2$ is similar). Then, by definition, $J_{2,1}(I_{2}(x_{1},x_{2}))=x_{1}$.
Since $I_{2}(x_{1},x_{2})\geq x_{i}$ for $i=1,2$ or, more generally,
$I_{k}(x_{1},\ldots,x_{k})\geq x_{i}$ for all $i=1,\ldots,k$ (see
e.g.~\cite[p.~27]{Odi89}), we have the following algorithm to compute
$J_{2,1}$, given some input $x\in\mathbb{N}$:

\begin{enumerate}
\item For all $i=1,\ldots,x$

\item For all $j=1,\ldots,x$

\item If $I_{2}(i,j)=x$, then output $i$

\item Next $j$

\item Next $i$
\end{enumerate}

This algorithm always stops with the correct result. Hence $J_{2,1}$ can be
computed by a one tape Turing machine $M$. Furthermore, using well-known
techniques, we can assume that $M$ has the following properties: (i) the tape
alphabet of the Turing machine is $\{B,1\}$ where $B$ denotes the blank
symbol; (ii) the input alphabet is $\{1\}$; (iii) each input $z\in\mathbb{N}$
and the respective output of the computation is represented in unary, i.e.~by
a sequence of $z$ 1's; and (iv) $J_{2,1}$ is computed by $M$ in time
$P(n)=P(x)$, where $P$ is a polynomial which can be explicitly obtained and
which is assumed to be an increasing function. Regarding condition (i), we
notice that there are universal Turing machines which only use the alphabet
$\{B,1\}$ and hence we do not lose computational computational power with
respect to Turing machines using more symbols. For example, if we have a
Turing machine $M_{1}$ which tape alphabet has $k>3$ symbols (including the
blank symbol $B$), then we can create a Turing machine $M_{2}$ with tape
alphabet $\{B,0,1\}$ which simulates $M_{1}$ by taking some fixed
$l\in\mathbb{N}$ satisfying $l\geq\log_{2}(k)$ such that each symbol of
$M_{1}$ is represented by distinct strings (blocks) of $\{0,1\}^{\ast}$ of
length $l$, with the exception of the blank symbol of $M_{1}$ which is
represented by a block of $l$ blank symbols in $M_{2}$. By its turn, $M_{2}$
can be simulated by a Turing machine $M_{3}$ with tape alphabet $\{B,1\}$ by
coding each symbol of $M_{2}$ as a string of length 2 in $M_{3}$, e.g.~by
coding $0,1$, and $B$ as $1B,11,$ and $BB$, respectively. We note that
regarding condition (iv), the expression of $P$ depends on the exact
implementation details of $M$, but we prefer to omit the exact description of
$M$ and of $P$ for brevity (these can be obtained as usual, although the
procedure is a bit tedious and hence not of much interest for this proof).

Let $g_{M}:\mathbb{R}^{7}\rightarrow\mathbb{R}^{6}$ be the function given by
Theorem \ref{Th:Simulation} such that $y^{\prime}=g_{M}(y)$ simulates $M$
(taking $\varepsilon=1/5$ in that theorem), and let $\eta<1/2$ be the
associated constant such that (\ref{Eq:ODE_bound}) holds. (The value of
$\delta$ in Theorem \ref{Th:Simulation} is not really relevant; one may simply
take $\delta=1/5)$.) Let $l\in\mathbb{N}$ be chosen such that $\sigma^{\lbrack
l]}(\eta)\leq1/5$ (see Lemma \ref{Lemma:sigma}). Given an input $x\in
\mathbb{N}$ for the Turing machine $M$, let us assume that this input is
encoded in unary (i.e.~$x$ is represented by a sequence of $x$ 1's) when
processed by $M$. We can then transform this unary coding of $x$ into another
integer value $\varphi_{1}(x)$ via the coding (\ref{Eq:ConfigEncoding}), where
$\gamma(B)=0$ and $\gamma(1)=1$, which can then be used to create an initial
condition for $y^{\prime}=g_{M}(y)$ such that this IVP simulates $M$ with
input $x$. Note that although $x\in\mathbb{N}$ and $\varphi_{1}(x)\in
\mathbb{N}$, we do not necessarily have $x=\varphi_{1}(x)$. Since initially
the tape will be empty, with the exception of the input, and $M$ will be on
its initial state, which we assume to be the state 1 (we can assume, without
loss of generality, that the states of $M$ correspond to the elements of
$\{1,2,\ldots,m\}$), then the initial configuration of $M$ will be coded as
$(\varphi_{1}(x),0,1)$. Let $\Phi(t,\varphi_{1}(x))$ denote the solution of
$y^{\prime}=g_{M}(y)$ with initial condition associated to the configuration
$(\varphi_{1}(x),0,1)$ and let $\pi_{i}^{k}:\mathbb{R}^{k}\rightarrow
\mathbb{R}$ be the projection $\pi_{i}^{k}(x_{1},\ldots,x_{k})=x_{i}$ for
$1\leq i\leq k$. Note that $\Phi$ is analytic and that $M$ computes $J_{2,1}$.
We will use these facts to create the function $\Omega_{2,1}$ to be defined as
$\Omega_{2,1}(x)=\varphi_{2}\circ\pi_{4}^{6}\circ\sigma^{\lbrack l]}\circ
\Phi(P(x+1),\varphi_{1}(x))$ for some analytic functions $\varphi_{1}
,\varphi_{2}$ yet to be defined, which essentially translate the value of
$x\in\mathbb{N}$ into the coding (\ref{Eq:ConfigEncoding}) of its unary
representation (case of $\varphi_{1}$) and, reciprocally, converts the coding
of the unary representation back to the number encoded by this representation
(note again that $x\in\mathbb{N}$ may not be equal to the number $\bar{x}
\in\mathbb{N}$ encoding the symbolic representation -- unary, binary, etc. --
of $x$ given by (\ref{Eq:ConfigEncoding})). Since $\left\vert x-n\right\vert
\leq1/5$ and $P$ is assumed to be increasing, it follows that $x+1>n\geq0$ and
$P(x+1)\geq P(n)$. Hence, if $\left\vert x-n\right\vert \leq1/5$ implies that
$\left\vert \varphi_{1}(x)-\varphi_{1}(n)\right\vert \leq1/5$, we get that
$\Phi(P(x+1),\varphi_{1}(x))$ will return the coding of the output of $M$ with
the input encoding the number $x\in\mathbb{N}$ (note that although the
relation (\ref{Eq:ODE_bound}) is in general valid only in intervals of the
format $[j,j+1/2]$ with $j\in\mathbb{N}$, but since we have assumed that the
image of an halting configuration is itself, it follows from the results of
\cite{GCB08} that (\ref{Eq:ODE_bound}) is valid for all times $[j+1/2,j+1]$
after the Turing machine has halted. See also Remark \ref{Remark:halting}). We
now only have to define the functions $\varphi_{1}$ and $\varphi_{2}$. Let us
now first turn our attention to $\varphi_{1}$. Note that given some
$n\in\mathbb{N}$, the number $2^{n}-1$ will represent $n$ in unary when using
the coding (\ref{Eq:ConfigEncoding}) (taking $k=2$, since $\gamma(B)=0$ by
definition, and by taking $\gamma(1)=1$). Hence it makes sense to take
$\varphi_{1}(n)=2^{n}-1$. However, we cannot take $\varphi_{1}(x)$ to be
$2^{x}-1$, because in that case we cannot ensure that $\left\vert
x-n\right\vert \leq1/5$ implies $\left\vert \varphi_{1}(x)-\varphi
_{1}(n)\right\vert \leq1/5$. To avoid this problem, we improve the accuracy of
$x$ using the function $\Psi$ from Lemma \ref{Lemma:Psi}, obtaining an
improved estimate $\bar{x}$ satisfying $\left\vert 2^{\bar{x}}-2^{n}
\right\vert \leq1/5$. We now determine the accuracy improvement needed to
achieve this objective. Note that the exponential function $2^{x}$ is strictly
increasing and thus, by the mean value theorem, we have
\[
\left\vert 2^{\bar{x}}-2^{n}\right\vert \leq2^{\max(\bar{x},n)}\ln2\left\vert
x-n\right\vert <2^{x+1}\left\vert \bar{x}-n\right\vert .
\]
Hence, if we have $\left\vert \bar{x}-n\right\vert \leq2^{-(x+4)}$, we get
$\left\vert 2^{\bar{x}}-2^{n}\right\vert \leq1/5$. This is achieved if
$\bar{x}=\Psi(x,x+2)$, due to Lemma \ref{Lemma:Psi} and from the property that
$\left\vert x-n\right\vert \leq1/5<2^{-2}$. Hence we can take $\varphi
_{1}(x)=2^{\Psi(x,x+2)}-1$.

We now proceed with a similar reasoning for $\varphi_{2}$. We first note that
if $n\in\mathbb{N}$ codes the exact output of $M$ according to
(\ref{Eq:ConfigEncoding}), and thus represents in unary some number
$i\in\mathbb{N}$, we will have $n=2^{i}-1$ as we have already seen. This
implies that $i=\log_{2}(n+1)$. Now we have to analyze again the effect of
replacing $n$ by some real value $x$ satisfying $\left\vert x-n\right\vert
\leq1/5$. By the mean value theorem, we have (note also that $n\geq0$ since
$n\in\mathbb{N}$)
\begin{align*}
\left\vert \log_{2}(x+1)-\log_{2}(n+1)\right\vert  &  \leq\frac{1}{\ln
2(\min(x,n)+1)}\left\vert (x+1)-(n+1)\right\vert \\
&  \leq\frac{1}{\ln2(n+4/5)}\left\vert x-n\right\vert \\
&  \leq\frac{5}{4\ln2}\left\vert x-n\right\vert \\
&  <2\left\vert x-n\right\vert .
\end{align*}
Therefore, to ensure that $\left\vert x-n\right\vert \leq1/5$ implies that
$\left\vert \log_{2}(x+1)-\log_{2}(n+1)\right\vert \leq1/5$, it is enough to
take $\varphi_{2}(x)=\log_{2}(\Psi(x,2)+1)$ (using Lemma \ref{Lemma:Psi})\ or
$\varphi_{2}(x)=\log_{2}(\sigma(x)+1)$ (using Lemma \ref{Lemma:sigma} and
noting, as mentioned in \cite[Remark 6]{GCB08}, that we can take
$\lambda_{1/4}=0.4\pi-1\approx0.2566371$).

Proceeding similarly for the case of $J_{2,2}$, we conclude that $\Omega
_{2,2}(x)=\varphi_{2}\circ\pi_{4}^{6}\circ\sigma^{\lbrack l]}\circ
\Phi_{M^{\prime}}(P_{M^{\prime}}(x+1),\varphi_{1}(x))$, where $M^{\prime}$ is
a TM machine computing $J_{2,2}$ which is similar to $M$, with the difference
that in Step 3 of the pseudo-algorithm above we take \textquotedblleft If
$I_{2}(i,j)=x$, then output $j$\textquotedblright. The results for $k>2$
follow inductively.
\end{proof}

\section{Analytic one-dimensional maps robustly simulate Turing
machines\label{Sec:Turing-map}}

We now present one of the main results of this paper. Let $M$ be a one-tape
Turing machine and let $(y_{1},y_{2},q)\in\mathbb{N}^{3}$ be the encoding of a
configuration as given in Section \ref{Sec:Dimension3} and
(\ref{Eq:ConfigEncoding}). In what follows each configuration $(y_{1}
,y_{2},q)$ is encoded in the single value
\[
c=C(y_{1},y_{2},q)=I_{3}(y_{1},y_{2},q)\in\mathbb{N}\text{.}
\]
Thus we can consider that, under this new encoding, the transition function of
a Turing machine is a map $\psi:\mathbb{N}\rightarrow\mathbb{N}$.

\begin{theorem}
\label{Th:one-dim-Turing-map}Let $\psi:\mathbb{N}\rightarrow\mathbb{N}$ be the
transition function of a Turing machine $M$, under the encoding described
above, and let $0\leq\delta<1/5$. Then there is an analytic function
$g_{M}:\mathbb{R}\rightarrow\mathbb{R}$ that robustly simulates $M$ in the
following sense: for all $g$ such that $\left\Vert g-g_{M}\right\Vert
\leq\delta$, and for all $\bar{x}_{0}\in\mathbb{R}$ satisfying $\left\vert
\bar{x}_{0}-x_{0}\right\vert \leq1/5,$ where $x_{0}\in\mathbb{N}$ represents
some configuration, one has for all $j\in\mathbb{N}$
\begin{equation}
\left\vert g^{[j]}(\bar{x}_{0})-\psi^{\lbrack j]}(x_{0})\right\vert \leq1/5.
\label{Eq:map-iteration}
\end{equation}

\end{theorem}

\begin{proof}
Most of the work to prove this theorem was already done in Section
\ref{Sec:auxiliary}. Let us first define a function $\bar{g}_{M}
:\mathbb{R}\rightarrow\mathbb{R}$ that robustly simulates $M$ in a weaker
sense that just the input can be perturbed and not $\bar{g}_{M}$ itself. More
specifically, let us define a function $\bar{g}_{M}:\mathbb{R}\rightarrow
\mathbb{R}$ with the following property:\ for all $\bar{x}_{0}\in\mathbb{R}$
satisfying $\left\vert \bar{x}_{0}-x_{0}\right\vert \leq1/5,$ where $x_{0}
\in\mathbb{N}$ represents some configuration, one has for all $j\in
\mathbb{N}$
\begin{equation}
\left\vert \bar{g}_{M}^{[j]}(\bar{x}_{0})-\psi^{\lbrack j]}(x_{0})\right\vert
\leq1/5. \label{Eq:aux3}
\end{equation}
To achieve this purpose, let $f_{M}$ be the corresponding 3-dimensional map
simulating $M$ obtained via Theorem \ref{Th:Main} with $\varepsilon=1/5$. Then
we take
\[
\bar{g}_{M}(x)=\Upsilon_{3}\circ f_{M}(\Omega_{3,1}(x),\Omega_{3,2}
(x),\Omega_{3,3}(x)).
\]
It then follows from Theorem \ref{Th:Main}, Proposition \ref{Prop:many-to-one}
, and Proposition \ref{Prop:one-to-many} that property (\ref{Eq:aux3}) is
satisfied. We now only have to take care of the perturbations to $\bar{g}_{M}
$.
Let $\sigma$ be the function defined in Lemma \ref{Lemma:sigma}. Let
$j\in\mathbb{N}$ be some integer such that $0<\lambda_{1/4}^{j}/5<1/5-\delta$
and take
\[
g_{M}(x)=\sigma^{\lbrack j]}\circ\bar{g}_{M}(x).
\]
Then, by property (\ref{Eq:aux3}) and Lemma \ref{Lemma:sigma}, if $\bar{x}
_{0}\in\mathbb{R}$ satisfies $\left\vert \bar{x}_{0}-x_{0}\right\vert
\leq1/5,$ where $x_{0}\in\mathbb{N}$ represents some configuration, one has
\[
\left\vert g_{M}(\bar{x}_{0})-\psi(x_{0})\right\vert \leq1/5-\delta.
\]
If $\left\Vert g-g_{M}\right\Vert \leq\delta$, then we conclude that
\begin{align*}
\left\vert g(\bar{x}_{0})-\psi(x_{0})\right\vert  &  \leq\left\vert g(\bar
{x}_{0})-g_{M}(\bar{x}_{0})\right\vert +\left\vert g_{M}(\bar{x}_{0}
)-\psi(x_{0})\right\vert \\
&  \leq\delta+(1/5-\delta)\\
&  \leq1/5.
\end{align*}
By using this last inequality and by iterating $g$ and $\psi$, we conclude
that the property (\ref{Eq:map-iteration}) holds.
\end{proof}

\begin{remark}
In the statement of Theorem \ref{Th:one-dim-Turing-map}, we could have picked
some fixed $\varepsilon>0$ satisfying $\delta<\varepsilon\leq1/5$ and, instead
of assuming that $\left\vert \bar{x}_{0}-x_{0}\right\vert \leq1/5$, we could
have assumed that $\left\vert \bar{x}_{0}-x_{0}\right\vert \leq\varepsilon$
and required that $\left\vert g^{[j]}(\bar{x}_{0})-\psi^{\lbrack j]}
(x_{0})\right\vert \leq\varepsilon$ for condition (\ref{Eq:map-iteration}). To
see this it would be enough to compose $g_{M}$ with $\sigma^{\lbrack l]}$,
where $\sigma$ is given by Lemma \ref{Lemma:sigma} and $l\in\mathbb{N}$ is
such that $\lambda_{1/4}^{l}/4\leq\varepsilon-\delta$, with $\lambda
_{1/4}=0.4\pi-1\approx0.2566371$.
\end{remark}

\section{Analytic two-dimensional ODEs can robustly simulate Turing
machines\label{Sec:ODE-simulation}}

In this section we construct an analytic two-dimensional ODE that robustly
simulates Turing machines in the sense of Theorem \ref{Th:Simulation}. To
prove this result, we simulate the iteration of the one-dimensional analytic
function provided by Theorem \ref{Th:one-dim-Turing-map} using a
two-dimensional analytic ODE. Then it follows from Theorem
\ref{Th:one-dim-Turing-map} that this ODE will simulate a TM. The approach is
similar to that used in \cite{GCB08}. More precisely, the following theorem is
proved in this section.

\begin{theorem}
\label{Th:two-dim-ODE-Turing}Let $\psi:\mathbb{N}\rightarrow\mathbb{N}$ be the
transition function of a Turing machine $M$, under the encoding described in
Section \ref{Sec:Turing-map}, and let $0\leq\delta<2/5$. Then there exist:

\begin{itemize}
\item $\eta>0$ satisfying $\eta<2/5<1/2$, which can be computed from $\delta$; and

\item an analytic function $g_{M}:\mathbb{R}^{3}\rightarrow\mathbb{R}^{2}$
\end{itemize}

\noindent such that the ODE $z^{\prime}=g_{M}(t,z)$ robustly simulates $M$ in
the following sense: for all $g$ satisfying $\left\Vert g-g_{M}\right\Vert
\leq\delta<2/5$ and for all $x_{0}\in\mathbb{N}$ which encodes a configuration
according to the encoding described above, if $\bar{x}_{0},\bar{y}_{0}
\in\mathbb{R}$ satisfy the conditions $\left\Vert \bar{x}_{0}-x_{0}\right\Vert
\leq1/5$ and $\left\Vert \bar{y}_{0}-x_{0}\right\Vert \leq1/5$, then the
solution $z(t)$ of
\[
z^{\prime}=g(t,z),\qquad z(0)=(\bar{x}_{0},\bar{y}_{0})
\]
satisfies, for all $j\in\mathbb{N}_{0}$ and for all $t\in\lbrack j,j+1/2]$,
\[
\left\Vert z_{2}(t)-\psi^{\lbrack j]}(x_{0})\right\Vert \leq\eta,
\]
where $z(t)\equiv(z_{1}(t),z_{2}(t))\in\mathbb{R}^{2}$.
\end{theorem}

The remaining of this section is devoted to the proof of Theorem
\ref{Th:two-dim-ODE-Turing}. We first present the main ideas in \cite{GCB08}
for simulating the iteration of a map defined \emph{over the integers}, which
admits a robust analytic real extension, using an analytic ODE. We begin with
a construction that uses an analytic ODE to approximate a value $b\in
\mathbb{R}$ in a finite (given) amount of time with some (given) accuracy.
This construction will be needed to derive the ODE simulating the iteration of
the map. Consider the following basic ODE
\begin{equation}
y^{\prime}=c(b-y)^{3}\phi(t), \label{Eq:ODE_target}
\end{equation}
which was already studied in \cite{Bra05}, \cite{CMC00}, \cite[Section
7]{GCB08}. The ODE can be easily solved by separating variables, which gives
rise to the following result.

\begin{lemma}
[\cite{GCB08}]\label{Lem:Target}Consider a point $b\in\mathbb{R}$ (the
\emph{target}), some $\gamma>0$ (the \emph{targeting error}), time instants
$t_{0}$ (\emph{departure time}) and $t_{1}$ (\emph{arrival time}), with
$t_{1}>t_{0}$, and a function $\phi:\mathbb{R\rightarrow R}$ with the property
that $\phi(t)\geq0$ for all $t\geq t_{0}$ and $\int_{t_{0}}^{t_{1}}
\phi(t)dt>0$. Then the IVP defined by (\ref{Eq:ODE_target}) (the
\emph{targeting equation}) with the initial condition $y(t_{0})=y_{0}$ and
\begin{equation}
c\geq\frac{1}{2\gamma^{2}\int_{t_{0}}^{t_{1}}\phi(t)dt}
\label{Eq:Target_Def_c}
\end{equation}
has the property that $\left\vert y(t)-b\right\vert <\gamma$ for $t\geq t_{1}
$, independently of the initial condition $y_{0}\in\mathbb{R}$.
\end{lemma}

However, since we wish the ODE simulating Turing machines to be robust to
perturbations, we have to analyze a perturbed version of (\ref{Eq:ODE_target}).

\begin{lemma}
[\cite{GCB08}]\label{Lem:Target:pertub} Consider a point $b\in\mathbb{R}$ (the
\emph{target}), some $\gamma>0$ (the \emph{targeting error}), time instants
$t_{0}$ (\emph{departure time}) and $t_{1}$ (\emph{arrival time}), with
$t_{1}>t_{0}$, and a function $\phi:\mathbb{R\rightarrow R}$ with the property
that $\phi(t)\geq0$ for all $t\geq t_{0}$ and $\int_{t_{0}}^{t_{1}}
\phi(t)dt>0$. Let $\rho,\delta\geq0$ and let $\bar{b},E:\mathbb{R\rightarrow
R}$ be functions with the property that $\left\vert \overline{b}
(t)-b\right\vert \leq\rho$ and $\left\vert E(t)\right\vert \leq\delta$ for all
$t\geq t_{0}$. Then the IVP defined by
\begin{equation}
z^{\prime}=c(\overline{b}(t)-z)^{3}\phi(t)+E(t), \label{Eq:Perturbed_target}
\end{equation}
with the initial condition $z(t_{0})=\bar{z}_{0}$, where $c$ satisfies
(\ref{Eq:Target_Def_c}), has the property that $\left\vert z(t_{1}
)-b\right\vert <\rho+\gamma+\delta(t_{1}-t_{0})$, independently of the initial
condition $\bar{z}_{0}\in\mathbb{R}$.
\end{lemma}

Proceeding along the lines of the argument presented in \cite{GCB08}, we now
show how the map given by Theorem \ref{Th:one-dim-Turing-map} can be iterated
by a 2-dimensional ODE. Although our objective is to obtain an analytic
function $g_{M}$ defining an ODE $z^{\prime}=g_{M}(t,z)$ which simulates a
given Turing machine $M$, in a first step we iterate the map given by Theorem
\ref{Th:one-dim-Turing-map} by using a non-analytic ODE.

Following the approach provided in \cite[p.~37]{Cam02b}, let $\theta
:\mathbb{R}\rightarrow\mathbb{R}$ be the $C^{\infty}$ function defined by
\begin{equation}
\theta(x)=0\text{ if }x\leq0,\text{ \ \ }\theta(x)=e^{-\frac{1}{x}}\text{ if
}x\geq0. \label{Eq:theta}
\end{equation}
Next we define the $C^{\infty}$ function $v:\mathbb{R}\rightarrow\mathbb{R}$
given by
\begin{equation}
v(0)=0,\text{ }v^{\prime}(x)=\bar{c}\theta(-\sin2\pi x), \label{Def:r}
\end{equation}
where
\[
\bar{c}=\left(  \int_{0}^{1}\theta(-\sin2\pi x)dx\right)  ^{-1}=\left(
\int_{1/2}^{1}e^{\frac{1}{\sin2\pi x}}dx\right)  ^{-1}
\]
The function $v$ has the property that $v(x)=n$, whenever $x\in\lbrack
0,n+1/2]$. We now get the following lemma.

\begin{lemma}
\label{Lem:r}The $C^{\infty}$ function $r:\mathbb{R}\rightarrow\mathbb{R}$
defined by $r(x)=v(x+1/4)$ has the property that $r(x)=n$, whenever
$x\in\lbrack n-1/4,n+1/4]$, for all integers $n$.
\end{lemma}

Note that the function $r$ can be seen as a function that returns the integer
part of a real number around a $1/4$-vicinity of an integer.

We now consider the ODE

\begin{equation}
\left\{
\begin{array}
[c]{rl}
z_{1}^{\prime} & =\tilde{c}(\tilde{f}(r(z_{2}))-z_{1})^{3}\theta(\sin2\pi
t),\\
z_{2}^{\prime} & =\tilde{c}(r(z_{1})-z_{2})^{3}\theta(-\sin2\pi t),
\end{array}
\right.  \label{iteration}
\end{equation}
where $\tilde{f}:\mathbb{R}\rightarrow\mathbb{R}$ is an extension of the
function $f:\mathbb{N}\rightarrow\mathbb{N}$, $z_{1}(0)=z_{2}(0)=x_{0}
\in\mathbb{N}$ and $\tilde{c}$ is a constant yet to be defined. We will next
show that (\ref{iteration}) iterates the map $f$ near integers. Its behavior
is depicted in Fig.\ \ref{fig:iteration-exp} when iterating the exponential
function $2^{x}$. Suppose that $t\in\lbrack0,1/2]$. Then $z_{2}^{\prime}(t)=0$
and thus $z_{2}(t)=x_{0}$ and $r(z_{2})=x_{0}$. In this manner, the first
equation of (\ref{iteration}) behaves like the targeting equation
(\ref{Eq:ODE_target}), where $b=\tilde{f}(r(z_{2}))=f(x_{0})$, $t_{0}=0$,
$t_{1}=1/2$, and $\phi(t)=\theta_{j}(\sin2\pi t)$ and $\tilde{c}$ has to
satisfy the condition (\ref{Eq:Target_Def_c}) for $c$. Now note that $\sin2\pi
t\geq1/\sqrt{2}$ and thus $-1/\sin(2\pi t)\geq-\sqrt{2}$ when $t\in
\lbrack1/8,3/8]$, which implies that
\begin{align*}
\int_{0}^{1/2}e^{-\frac{1}{\sin2\pi t}}dt  &  \geq\int_{1/4}^{2/8}e^{-\frac
{1}{\sin2\pi t}}dt\\
&  \geq\frac{1}{4}e^{-\sqrt{2}}\text{.}
\end{align*}
This implies that
\begin{align*}
\frac{1}{2\gamma^{2}\int_{0}^{1/2}\theta(\sin2\pi t)dt}  &  =\frac{1}
{2\gamma^{2}\int_{0}^{1/2}e^{-\frac{1}{\sin2\pi t}}dt}\\
&  \leq\frac{1}{2\gamma^{2}\int_{0}^{1/2}e^{-\frac{1}{\sin2\pi t}}dt}\\
&  \leq\frac{2e^{\sqrt{2}}}{\gamma^{2}}.
\end{align*}
Therefore, due to Lemma \ref{Lem:Target} and (\ref{Eq:Target_Def_c}), if we
take $\tilde{c}\geq2e^{\sqrt{2}}/\gamma^{2}$ the first equation of
(\ref{iteration}) becomes a targetting equation on the time interval
$[0,1/2]$\ associated to a targetting error $\gamma$. In particular, if we
pick $\gamma=1/5$, then we can pick $\tilde{c}=206$ such that
(\ref{Eq:Target_Def_c}) holds and thus
\begin{figure}
    \begin{center}
        \includegraphics[width=6cm]{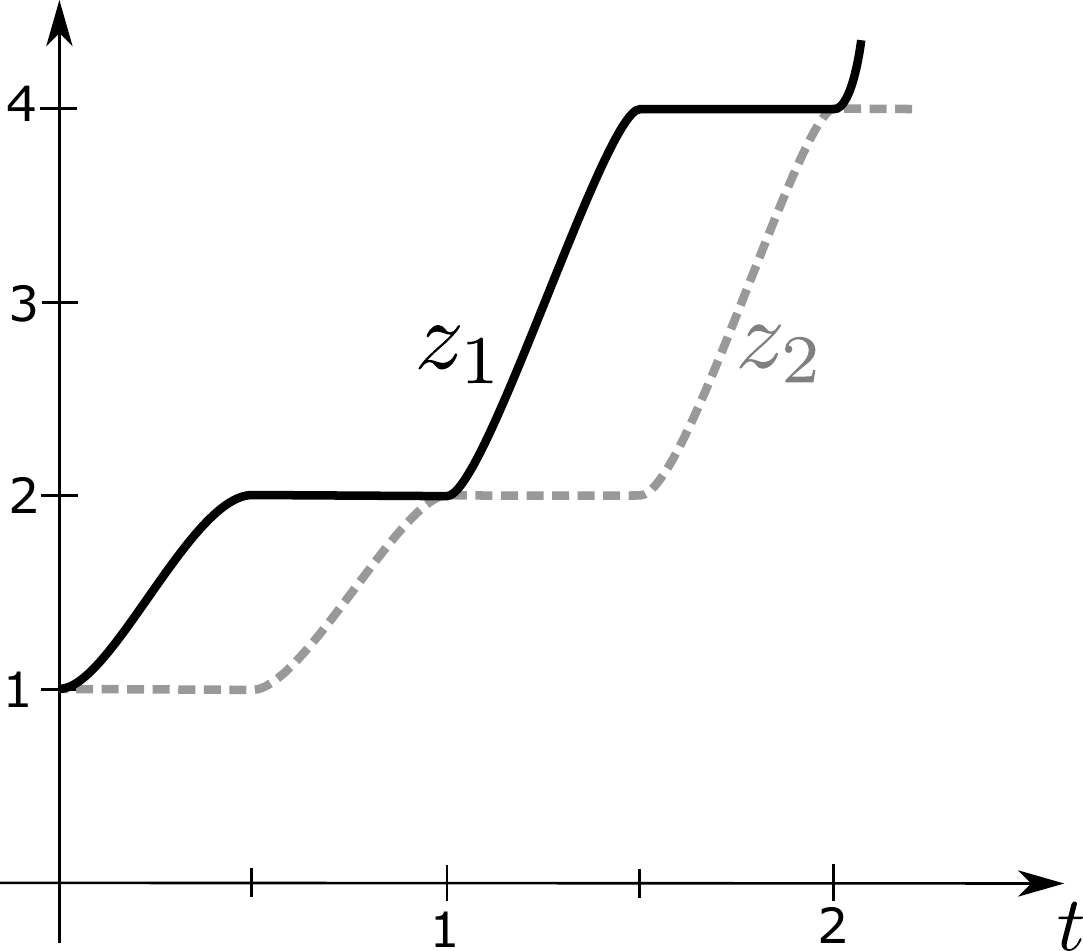}
        \caption{Iterating the exponential function $2^x$ with an ODE.}
        \label{fig:iteration-exp}
    \end{center}
\end{figure}
\[
\left\vert b-z_{1}(1/2)\right\vert =\left\vert f(x_{0})-z_{1}(1/2)\right\vert
\leq1/5\text{.}
\]
On the time interval $[1/2,1]$, the roles of $z_{1}$ and $z_{2}$ are switched:
we will have $z_{1}^{\prime}(t)=0$ which implies that $z_{1}(t)=z_{1}(1/2)$
for all $t\in\lbrack1/2,1]$. We thus conclude that $r(z_{1}(t))=f(x_{0})$ for
$t\in\lbrack1/2,1]$ and therefore the second equation of (\ref{iteration})
behaves like the targeting equation (\ref{Eq:ODE_target}) with $b=r(z_{1}
(t))=f(x_{0})$, $t_{0}=1/2$, $t_{1}=1$, and $\phi(t)=\theta_{j}(-\sin2\pi t)$
and $\tilde{c}=206$. Again, using Lemma \ref{Lem:Target} and similar arguments
as in the previous case, we conclude that
\[
\left\vert b-z_{2}(1)\right\vert =\left\vert f(x_{0})-z_{2}(1)\right\vert
\leq1/5\text{.}
\]
In the following time interval $[1,3/2]$, the cycle repeats itself and we have
$z_{2}^{\prime}(t)=0$ and thus $\tilde{f}(r(z_{2}))=f(f(x_{0}))=f^{[2]}
(x_{0})$. Using a similar reasoning, we conclude that for all $k\in
\mathbb{N}_{0}$ we have $z_{2}(t)=f^{[k]}(x_{0})$ for all $t\in\lbrack
k,k+1/2]$, (assuming $f^{[0]}(x)=x$) and $z_{1}(t)=f^{[k+1]}(x_{0})$ for all
$t\in\lbrack k+1/2,k+1]$.

We thus have shown how to iterate (the extension of) a discrete function
$f:\mathbb{N}\rightarrow\mathbb{N}$ with an ODE. Nonetheless, the ODE is still
not analytic as required. Remark that if the ODE is analytic, then
$z_{1}^{\prime}$ and $z_{2}^{\prime}$ cannot be $0$ in half-unit intervals,
since it is well-known that if an analytic function ($z_{1}^{\prime}$ and
$z_{2}^{\prime}$ in our case) takes the value zero in a non-empty interval,
then this function has to be identically equal to $0$ on all its domain.
Therefore, instead of requiring that $z_{1}^{\prime}$ and $z_{2}^{\prime}$
take the value $0$ in alternating half-unit intervals, we require that these
functions take values \emph{very close} to zero. Since the map $g_{M}$ given
by Theorem \ref{Th:one-dim-Turing-map} is robust to perturbations on its
input, this will ensure that the whole simulation of $g_{M}$ with a
two-dimensional ODE can still be performed, even if $z_{1}$ and $z_{2}$ are
not strictly constant in the half-intervals $[k+1/2,k+1]$ and $[k,k+1/2]$,
respectively. However, we still have to ensure that $\left\vert z_{1}^{\prime
}(t)\right\vert $ and $\left\vert z_{2}^{\prime}(t)\right\vert $ are
sufficiently small in the half-unit intervals of interests to guarantee that
the iteration can be carried faithfully. To better understand how this can be
achieved, we have to analyze the effects of introducing perturbations in
(\ref{iteration}), with the help of Lemma \ref{Lem:Target:pertub} since now
the \textquotedblleft targets\textquotedblright\ will be slightly perturbed.

Proceeding as in \cite{GCB08}, the non-analytic function $\theta_{j}(\sin2\pi
t)$ in the first equation of (\ref{iteration}) is replaced by an analytic
periodic function with period 1 which is close to zero when $t\in
\lbrack1/2,1]$. As shown in \cite{GCB08}, this can be done by considering the
function $s$ defined by
\begin{equation}
s(t)=\frac{1}{2}\left(  \sin^{2}(2\pi t)+\sin(2\pi t)\right)  .
\label{Eq:Def_s}
\end{equation}
which ranges between $0$ and $1$ in $[0,1/2]$ (and, in particular, between
$4/5$ and $1$ when $x\in\lbrack0.17,0.33]$), and between $-\frac{1}{8}$ and
$0$ on the time interval $[1/2,1]$. Then we take the analytic function
$\phi:\mathbb{R}^{2}\rightarrow\lbrack0,1]$ defined by
\[
\phi(t,y)=\Psi(s(t),y),
\]
we conclude that $\int_{0}^{1/2}\phi(t,y)dt>4/5\times(0.33-0.17)=0.128>0$
(assuming that $y\geq5)$ and $\left\vert \phi(t,y)\right\vert <e^{-y}/8$ for
all $t\in\lbrack1/2,1]$ (i.e.~$y$ allows us to provide an error bound for
$z_{1}^{\prime}(t)$ in the time interval $[1/2,1]$). Since $\phi$ has period
$1$ on $t$, we conclude that $\int_{k}^{k+1/2}\phi(t,y)dt>0.128>0$ and
$\left\vert \phi(t,y)\right\vert <e^{-y}/8$ for all $t\in\lbrack k+1/2,k+1]$,
where $k\in\mathbb{N}$ is arbitrary and $y\geq5$. Therefore $\phi$ satisfies
the assumptions of the function $\phi$ in Lemma \ref{Lem:Target:pertub} on the
time interval $[0,1/2]$.

We can now proceed with the main construction that simulates the
iteration of the map given by Theorem \ref{Th:one-dim-Turing-map}
with an analytic ODE. Take $\gamma>0$ to be a value such that
$2\gamma+\delta/2\leq1/5$, and let $g_{M}$ be the map given by
Theorem \ref{Th:one-dim-Turing-map} (use as value for $\delta$ in
the statement of the Theorem \ref{Th:one-dim-Turing-map} the value
$\eta/2<1/5$, where $\eta=\left(  \gamma+\delta\right) /2+1/5<2/5$).
Consider the ODE $z^{\prime}=h_{M}(t,z)$ given by
\begin{align}
z_{1}^{\prime}  &  =c_{1}(z_{1}-\sigma^{\lbrack l]}\circ g_{M}\circ
\sigma^{\lbrack l]}(z_{2}))^{3}\,\phi_{1}(t,z_{1},z_{2}),
\label{Eq:Simulation}\\
z_{2}^{\prime}  &  =c_{2}(z_{2}-\sigma^{\lbrack l]}(z_{1}))^{3}\,\phi
_{2}(t,z_{1},z_{2}),\nonumber
\end{align}
where $z(t)=(z_{1}(t),z_{2}(t))\in\mathbb{R}^{2}$, with initial conditions
$z_{1}(0)=\bar{x}_{0}$, $z_{2}(0)=\bar{y}_{0}$, where $\bar{x}_{0},\bar{y}
_{0}\in\mathbb{R}$ are approximations of some initial configuration $x_{0}$
satisfying $\left\vert \bar{x}_{0}-x_{0}\right\vert \leq1/5$ and $\left\vert
\bar{y}_{0}-x_{0}\right\vert \leq1/5$, $l\in\mathbb{N}$ is such that
$\sigma^{\lbrack l]}(\eta)\leq\gamma<1/5$ (see Lemma \ref{Lemma:sigma}), and
\begin{align}
\phi_{1}(t,z_{1},z_{2})  &  =\phi\left(  t,\tfrac{c_{1}}{\gamma}(z_{1}
-\sigma^{\lbrack l]}\circ g_{M}\circ\sigma^{\lbrack l]}(z_{2}))^{4}
+\tfrac{c_{1}}{\gamma}+10\right)  ,\nonumber\\
\phi_{2}(t,z_{1},z_{2})  &  =\phi\left(  -t,\tfrac{c_{2}}{\gamma}(z_{2}
-\sigma^{\lbrack l]}(z_{1}))^{4}+\tfrac{c_{2}}{\gamma}+10\right)  ,
\label{SistemaCont}
\end{align}
and $c_{1},c_{2}$ are constants associated to a targeting error of value
$\gamma$ (i.e.~they are chosen such as the constant $c$ in
(\ref{Eq:Target_Def_c}) and by noting that $\int_{0}^{1/2}\phi(t,y)>0.128$, we
can take $c_{1},c_{2}=4\gamma^{-2}$). Note that, since $\phi$ satisfies the
assumptions of function $\phi$ in Lemma \ref{Lem:Target:pertub}, the same will
happen to $\phi_{1}$ and $\phi_{2}$. The ODE (\ref{Eq:Simulation}) simulates
the Turing machine $M$ similarly to the ODE (\ref{iteration}), by iterating
$g_{M}$. However (\ref{SistemaCont}) has a more complicated expression, since
it has to deal with the fact that $z_{1}^{\prime}(t)$ and $z_{2}^{\prime}(t)$
are not exactly zero in half-unit intervals.

Since we want that the ODE $z^{\prime}=h_{M}(t,z)$ to robustly simulate $M$,
let us assume that the right hand-side of the equations in
(\ref{Eq:Simulation}) can be subject to an error of absolute value not
exceeding $\delta$. To start the analysis of the behavior of
(\ref{Eq:Simulation}), let us first consider the time interval $[0,1/2]$.
Since$\ \left\vert x\right\vert ^{3}\leq x^{4}+1$ for all $x\in\mathbb{R}$, we
conclude that $\phi_{2}$ is less than $\min(\gamma(c_{2}\left\vert {z}
_{2}-\sigma^{\lbrack n_{2}]}({z}_{1})\right\vert ^{3})^{-1},1/10)$. This
implies, together with the assumption that $z_{2}^{\prime}$ in
(\ref{SistemaCont}) is perturbed by an amount not exceeding $\delta$, that
$\left\vert z_{2}^{\prime}(t)\right\vert \leq\gamma+\delta$ for all
$t\in\lbrack0,1/2]$ which implies that $\left\vert z_{2}(t)-z_{2}
(0)\right\vert \leq(\gamma+\delta)/2$ for all $t\in\lbrack0,1/2]$. Since the
initial condition $\bar{x}_{0}$ satisfies $\left\vert \bar{x}_{0}
-x_{0}\right\vert \leq1/5$ where $x_{0}\in\mathbb{N}$, we conclude that
\[
\left\vert {z}_{2}(t)-x_{0}\right\vert \leq\frac{\gamma+\delta}{2}
+1/5=\eta<\frac{2}{5}\text{ \ \ \ for all }t\in\lbrack0,1/2]
\]
which implies that
\[
\left\vert \sigma^{\lbrack l]}(z_{2}(t))-x_{0}\right\vert \leq\gamma<\frac
{1}{5}\text{ \ \ \ for all }t\in\lbrack0,1/2].
\]
Due to Theorem \ref{Th:one-dim-Turing-map}, we conclude that
\[
\left\vert g_{M}\circ\sigma^{\lbrack l]}(z_{2}(t))-\psi(x_{0})\right\vert
\leq\frac{1}{5}\text{ \ \ \ for all }t\in\lbrack0,1/2].
\]
Since $\sigma^{\lbrack l]}(\eta)\leq\gamma$ and $\eta>1/5$, we get that
$\left\vert \sigma^{\lbrack l]}\circ g_{M}\circ\sigma^{\lbrack l]}
(z_{2}(t))-g_{M}(x_{0})\right\vert <\gamma$ for all $t\in\lbrack0,1/2]$. Then
the behavior of $z_{1}$ is given by Lemma \ref{Lem:Target:pertub} and
\begin{equation}
\left\vert {z}_{1}\left(  \frac{1}{2}\right)  -\psi(x_{0})\right\vert
<2\gamma+\delta/2\leq\frac{1}{5}. \label{Eq:aux4}
\end{equation}

In the next half-unit interval $[1/2,1]$ the roles of ${z}_{1}$ and ${z}_{2}$
are switched as before and one concludes, using similar arguments to those
used in the time interval $[0,1/2]$ that $\left\vert z_{1}^{\prime
}(t)\right\vert \leq\gamma+\delta$ for $t\in\lbrack1/2,1]$. Therefore
$\left\vert z_{1}(t)-z_{1}(1/2)\right\vert \leq(\gamma+\delta)/2$ for all
$t\in\lbrack1/2,1]$. This inequality together with (\ref{Eq:aux4}) yields that
\[
\left\vert {z}_{1}(t)-\psi(x_{0})\right\vert <1/5+(\gamma+\delta)/2=\eta\text{
\ \ \ for all }t\in\lbrack1/2,1].
\]
Since $\sigma^{\lbrack l]}(\eta)\leq\gamma$, we conclude that $\left\vert
\sigma^{\lbrack l]}({z}_{1}(t))-\psi(x_{0})\right\vert \leq\gamma$ and Lemma
\ref{Lem:Target:pertub} gives us
\begin{equation}
\left\vert {z}_{2}\left(  1\right)  -\psi(x_{0})\right\vert <2\gamma
+\delta/2\leq\frac{1}{5}. \label{Eq:aux5}
\end{equation}
On the time interval $[1,3/2]$, the roles of $z_{1}$ and $z_{2}$ are again
switched. There we will have $\left\vert z_{2}^{\prime}(t)\right\vert
\leq\gamma+\delta$ for all $t\in\lbrack1,3/2]$ which implies that $\left\vert
z_{2}(t)-z_{2}(1)\right\vert \leq(\gamma+\delta)/2$ for all $t\in
\lbrack1/2,1]$. This inequality together with (\ref{Eq:aux5})\ gives us
\[
\left\vert {z}_{2}(t)-\psi(x_{0})\right\vert <1/5+(\gamma+\delta)/2=\eta\text{
\ \ \ for all }t\in\lbrack1,3/2].
\]
We thus conclude that this analysis can be repeated in subsequent time
intervals and thus that for all $j\in\mathbb{N}$, if $t\in\lbrack j,j+\frac
{1}{2}]$ then $\left\vert z_{2}(t)-\psi^{\lbrack j]}(x_{0})\right\vert
\leq1/5$. This concludes the proof.

\begin{remark}
\label{Remark:halting}From the proof of Theorem \ref{Th:two-dim-ODE-Turing},
we conclude that if the Turing machine halts after $n_{0}$ steps with
configuration $c_{h}$ and if we assume that $\psi(c_{h})=c_{h}$, then for any
$n\in\mathbb{N}$ satisfying $n\geq n_{0}+1$, we have
\[
\left\vert {z}_{2}\left(  n\right)  -c_{h}\right\vert \leq\frac{1}{5}.
\]
Furthermore, on the half-unit time interval $[n,n+1/2]$, we will get that
\[
\left\vert \sigma^{\lbrack l]}({z}_{1}(t))-c_{h}\right\vert <\gamma<1/5
\]
for $t\in\lbrack n,n+1/2]$. Assuming that $\delta=0$ in Theorem
\ref{Th:two-dim-ODE-Turing}, we then conclude from (\ref{Eq:Simulation}) that
$z_{2}(t)$ starts its trajectory from a value $1/5$-near to $c_{h}$ and will
monotically converge to a value $\sigma^{\lbrack l]}({z}_{1}(t))$ which,
although may change with time, will never leave an $1/5$-vicinity of $c_{h}$
for $t\in\lbrack n,n+1/2]$. This implies that
\begin{equation}
\left\vert {z}_{2}\left(  t\right)  -c_{h}\right\vert \leq\frac{1}{5}
\label{Eq:error}
\end{equation}
for all $t\in\lbrack n,n+1/2]$. Since (\ref{Eq:error}) holds on $[n,n+1/2]$ as
we have seen from the proof of Theorem \ref{Th:two-dim-ODE-Turing}, we
conclude that (\ref{Eq:error}) holds for $t\geq n_{0}+1/2$.
\end{remark}

\section{Simulating Turing machines over compact sets\label{Sec:Compact}}

So far all our simulations of Turing machines are carried out over
the non-compact set $\mathbb{R}^{n}$. In this section we show that
it is possible to simulate Turing machines with ODEs on the
2-dimensional compact sphere
$\mathbb{S}^{2}=\{x\in\mathbb{R}^{3}:\left\Vert x\right\Vert =1\}$.
The technique used in the construction is based on a similar result
proved in \cite{CMP21}. It is shown in \cite{CMP21} that there is a
computable polynomial vector field on the sphere
$\mathbb{S}^{n}=\{x\in\mathbb{R}^{n+1}:\left\Vert x\right\Vert =1\}$
simulating a universal Turing machine; in other words, the
polynomial vector field is Turing complete, where $n\geq17$ is
arbitrary. The underlying idea of the construction presented in
\cite{CMP21} is to map a vector field defined in $\mathbb{R}^{n}$ to
$\mathbb{S}^{n}$ using the stereographic projection, and to remove
the singularity at the north pole by using a suitable
reparametrization on the polynomial vector fields. We show in this
section that the dimension can be lowered from $n\geq 17$ to $n=2$
at the cost of using a non-polynomial $C^{\infty}$ vector field.

The notion introduced in the following definition will be used
throughout this section.

\begin{definition}
Let $f:\mathbb{R}^{n+k}\rightarrow\mathbb{R}^{j}$, where $n,k,j\in
\mathbb{N}_{0}$, with $j\geq1$. Given an expression $f(x_{1},\ldots
,x_{n},y_{1},\ldots,y_{k})$, we say that $f$ is bounded by a constant on the
variables $x_{1},\ldots,x_{n}$ and bounded by a polynomial on the variables
$y_{1},\ldots,y_{k}$ if
\[
\left\Vert f(x_{1},\ldots,x_{n},y_{1},\ldots,y_{k})\right\Vert \leq
p(y_{1},\ldots,y_{k})
\]
for all $(x_{1},\ldots,x_{n},y_{1},\ldots,y_{k})\in\mathbb{R}^{n+k}$ in the
domain of $f$.
\end{definition}

For example $g(t,x)=x\sin(2\pi t)$ is bounded by a constant on $t$ and
polynomially bounded on $x$ and $h(t)=\sin(2\pi t)$ is bounded by a constant
on $t$ (to simplify notation, in this latter case where $h$ is not
polynomially bounded on any other variables, we will just say that $h$ is
bounded by a constant). Some functions which are bounded by a constant include
$\sin$, $\cos$, $\arctan$.

The next result shows when a Turing universal vector field can be
defined on $\mathbb{S}^{n}$.

\begin{theorem}
\label{Th:compact_simulation}Let
\begin{equation}
y^{\prime}=f(t,y) \label{Eq:ODE}
\end{equation}
be a $C^{1}$-computable ODE simulating a Turing machine $M$, where
$f:\mathbb{R}^{n+1}\rightarrow\mathbb{R}^{n}$ is a $C^{\infty}$
function with the property that both $f$ and all its partial
derivatives $\frac
{\partial^{\left\vert k\right\vert }f}{\partial t^{k_{0}}\partial y_{1}
^{k_{1}}\ldots\partial y_{n}^{k_{n}}}$ are bounded by a constant on
$t$ and are polynomially bounded on the variables
$x_{1},\ldots,x_{n}$, assuming that the argument of $f$ is $(t,x)\in\mathbb{R}^{n+1}$, where
$k_{0},k_{1},\ldots,k_{n}\in\mathbb{N}_{0}$,
$k=(k_{0},k_{1},\ldots,k_{n})$, and $\left\vert k\right\vert
=k_{0}+k_{1}+\ldots+k_{n}$. Then from $f$ one can compute  a
$C^{\infty}$ vector field $F$ defined on $\mathbb{S}^{n}$ that also
simulates $M$.
\end{theorem}

\begin{proof}
We recall that the (inverse) stereographic projection $\varphi:\mathbb{R}
^{n}\rightarrow\mathbb{S}^{n}\subseteq\mathbb{R}^{n+1}$ is given by
\[
\varphi(x_{1},x_{2},\ldots,x_{n})=\left(  \frac{r^{2}-1}{1+r^{2}},\frac
{2x_{1}}{1+r^{2}},\frac{2x_{2}}{1+r^{2}},\ldots,\frac{2x_{n}}{1+r^{2}}\right)
\]
where $r^{2}=x_{1}^{2}+x_{2}^{2}+\ldots+x_{n}^{2}$. Suppose that
$f(t,x)=(f_{1}(t,x),\ldots,f_{n}(t,x))$. Then we can write the vector field
defined by $f$ on $\mathbb{R}^{n}$ as
\[
f=\sum_{i=1}^{n}f_{i}\frac{\partial}{\partial x_{i}}.
\]
We recall that if $g:M\rightarrow N$ is a $C^{1}$ map between two manifolds
$M=\mathbb{R}^{n}$ and $N=\mathbb{R}^{k}$, then for each $p\in M$ the map $g$
induces a linear map $g_{\ast}:T_{p}M\rightarrow T_{g(p)}N$ from the tangent
space $T_{p}M$ of $M$ at $p$ to the tangent space of $N$ at $g(p)$. We also
recall that $(\left.  \partial/\partial x_{1}\right\vert _{p},\ldots,\left.
\partial/\partial x_{n}\right\vert _{p})$ forms a basis for $T_{p}M\ $and,
similarly, if $(\bar{x}_{1},\ldots,\bar{x}_{k})$ are coordinates for
$N=\mathbb{R}^{k}$, then $(\left.
\partial/\partial\bar{x}_{1}\right\vert _{g(p)},\ldots,\left.
\partial/\partial\bar{x}_{k}\right\vert _{g(p)})$ forms a basis for
$T_{g(p)}N$. Moreover, the matrix that (locally) defines the linear
map $g_{\ast}$, relative to the bases $(\left.  \partial/\partial
x_{1}\right\vert _{p},\ldots,\left.  \partial/\partial
x_{n}\right\vert _{p})$ and $(\left.
\partial/\partial\bar{x}_{1}\right\vert _{g(p)},\ldots,\left.
\partial/\partial\bar{x}_{k}\right\vert _{g(p)})$,\ is the Jacobian of $g$. In
the case of the map $\varphi$, and if we take
$(y_{0},y_{1},\ldots,y_{n})$ as coordinates for $\mathbb{R}^{n+1}$,
we obtain the following (note that the variable $t$ in
the expression of $f$ can be seen as a fixed parameter):
\begin{align}
\varphi_{\ast}\left(  f\right)   &  =\sum_{i=1}^{n}f_{i}\varphi_{\ast}\left(
\frac{\partial}{\partial x_{i}}\right) \nonumber\\
&  =\sum_{i=1}^{n}f_{i}\sum_{j=0}^{n}\frac{\partial\varphi_{j}}{\partial
x_{i}}\frac{\partial}{\partial y_{j}}\nonumber\\
&  =\sum_{i=1}^{n}f_{i}\cdot\left(  (1-y_{0})y_{i}\frac{\partial}{\partial
y_{0}}+(1-y_{0}-y_{i}^{2})\frac{\partial}{\partial y_{i}}-\sum
_{\substack{j=0\\j\notin\{0,i\}}}^{n}y_{i}y_{j}\frac{\partial}{\partial y_{j}
}\right)  , \label{Eq:phi_star}
\end{align}
where $f_{i}$ is evaluated at $(t,\varphi^{-1}(y_{0},y_{1},\ldots
,y_{n}))=\left(  t,\frac{y_{1}}{1-y_{0}},\ldots,\frac{y_{n}}{1-y_{0}}\right)
$. This implies that $\varphi_{\ast}\left(  f\right)  $ is a vector field of
class $C^{\infty}$, except at the north pole $y_{NP}=(1,0,\ldots,0)$ of
$\mathbb{S}^{n}$ where it is not defined. Let us now consider the following
ODE
\begin{align}
\tau^{\prime}  &  =K(\bar{x})\label{Eq:ODE_parameterization}\\
\bar{x}^{\prime}  &  =f(\tau,\bar{x})K(\bar{x})\nonumber
\end{align}
where $K:\mathbb{R}^{n}\rightarrow\mathbb{R}$ is such that $K(x)>0$ for any
$x\in\mathbb{R}^{n}$ and $\tau(0)=0$, $\bar{x}(0)=x_0$. Since $\tau^{\prime}(t)>0$, $\tau$ is
strictly increasing and thus $\tau$ admits an inverse $\tau^{-1}$. Next we
note that if $\tilde{x}=x\circ\tau$, where $x$ is a solution of (\ref{Eq:ODE})
and $\tau(0)=0$, we have that
\begin{align}
\tilde{x}^{\prime}(t)  &  =\left(  x(\tau(t))\right)  ^{\prime}\nonumber\\
&  =f(\tau(t),x(\tau(t)))\tau^{\prime}(t)\nonumber\\
&  =f(\tau(t),\tilde{x}(t))K(\bar{x}(t)). \label{Eq:ODE_parameterization_aux}
\end{align}
It is not difficult to see that if $(\tau,\bar{x})$ is a solution for
(\ref{Eq:ODE_parameterization}), then $\bar{x}$ is also a solution to
(\ref{Eq:ODE_parameterization_aux}). Since the solution of the ODE
(\ref{Eq:ODE_parameterization}) is unique, by the Picard-Lindel\"{o}f theorem,
we conclude that $\bar{x}(t)=\tilde{x}(t)=x(\tau(t))$. Furthermore, since
$\tau^{\prime}(t)>0$ for any $t\in\mathbb{R}$, we conclude that any solution
curve of (\ref{Eq:ODE}) with initial condition $y(0)=x_0$ also provides a solution curve for the last $n$ components of the solution of
(\ref{Eq:ODE_parameterization}) with initial condition $\tau(0)=0$, $\bar{x}(0)=x_0$, up to some time reparametrization, and vice versa.

Thus, by taking $K(x)=e^{-\frac{2}{1+r^{2}}}=e^{-\frac{2}{1+x_{1}^{2}
+\ldots+x_{n}^{2}}}$, we conclude that the solution curves of
\begin{align}
\tau^{\prime}  &  =e^{-\frac{2}{1+r^{2}}}\nonumber\\
x^{\prime}  &  =e^{-\frac{2}{1+r^{2}}}f(\tau,x)=h(t,x)
\label{Eq:ODE_new_simulation}
\end{align}
and of (\ref{Eq:ODE}) are the same, up to a time reparametrization $\tau$
given by the ODE $\tau^{\prime}=e^{-\frac{2}{1+r^{2}}}>0$. Note that the
right-hand side of (\ref{Eq:ODE_new_simulation}) is formed by $C^{1}
$-computable functions, which means that the solution $(\tau,x)$ to
(\ref{Eq:ODE_new_simulation}) is also computable, since the solution to a
$C^{1}$-computable system is also computable \cite{GZB07}, \cite{CG08}. Hence,
when simulating Turing machines, if the result of the $n$th step of the
computation of the Turing machine being simulated by (\ref{Eq:ODE}) can be
read in the time interval $[a_{n},b_{n}]$, then the result of the $n$th step
when the simulation is performed by (\ref{Eq:ODE_new_simulation}) can be read
on the time interval $[\tau^{-1}(a_{n}),\tau^{-1}(b_{n})]$. Note that
$\tau^{-1}$ can be computed from $\tau$ and hence from $f$. Indeed, we know
that the derivative of $\tau^{-1}$ is given by
\begin{align*}
\left(  \tau^{-1}(a)\right)  ^{\prime}  &  =\frac{1}{\tau^{\prime}(\tau
^{-1}(a))}\\
&  =\frac{1}{K(\bar{x}(\tau^{-1}(a)))}\\
&  =\frac{1}{K(x(\tau\circ\tau^{-1}(a)))}\\
&  \frac{1}{K(x(a))}.
\end{align*}
Hence $\tau^{-1}$ can be obtained as the solution of the initial-value problem
(IVP) defined by $\left(  \tau^{-1}(t)\right)  ^{\prime}=1/K(x(t))$,
$\tau^{-1}(0)=0$. Since the right-hand side of the ODE defining this IVP is
computable from $x$ and hence from $f$, we conclude that $\tau^{-1}$ is
computable from $f$ \cite{GZB07}, \cite{CG08}. In particular, if $f$ is
computable, then so is $\tau^{-1}$.

We now have (we can again assume that $t$ is a fixed parameter for
$h$,
where $h$ is given by (\ref{Eq:ODE_new_simulation}))
\[
\varphi_{\ast}\left(  h\right)  =e^{-(1-y_{0})}\varphi_{\ast}\left(  f\right)
.
\]
From (\ref{Eq:phi_star}) we conclude that (note that $\left\vert
y_{i}\right\vert \leq1$)
\begin{align*}
\left\Vert \varphi_{\ast}\left(  h\right)  (y)\right\Vert  &  =\left\Vert
e^{-(1-y_{0})}\varphi_{\ast}\left(  f\right)  (y)\right\Vert \\
&  =\left\Vert e^{-(1-y_{0})}\right\Vert \left\Vert \varphi_{\ast}\left(
f\right)  (y)\right\Vert \\
&  =\left\Vert e^{-(1-y_{0})}\right\Vert \left\Vert \sum_{i=1}^{n}f_{i}\left(
t,\frac{y_{1}}{1-y_{0}},\ldots,\frac{y_{n}}{1-y_{0}}\right)  \right.  \cdot\\
&  \cdot\left.  \left(  (1-y_{0})y_{i}\frac{\partial}{\partial y_{0}}
+(1-y_{0}-y_{i}^{2})\frac{\partial}{\partial y_{i}}-\sum
_{\substack{j=0\\j\notin\{0,i\}}}^{n}y_{i}y_{j}\frac{\partial}{\partial y_{j}
}\right)  \right\Vert \\
&  \leq e^{-(1-y_{0})} \cdot6n\cdot p\left(  \frac{y_{1}}{1-y_{0}}
,\ldots,\frac{y_{n}}{1-y_{0}}\right)  .
\end{align*}
This latter result implies that
\begin{equation}
\lim_{\substack{y\rightarrow y_{NP}\\y\in\mathbb{S}^{n}-\{y_{NP}\}}
}\varphi_{\ast}\left(  h\right)  (y)=0. \label{Eq:limit}
\end{equation}
Similarly, if we assume that
\[
\varphi_{\ast}\left(  h\right)  (y)=e^{-(1-y_{0})}\sum_{i=0}^{n}
s_{i}(t,y)\frac{\partial}{\partial y_{i}}
\]
where $s_{i}$ are functions, from (\ref{Eq:phi_star}) and from the assumption
that all the partial derivatives of $f$ of the form $\frac{\partial^{|k|}
f}{\partial t^{k_{0}}\partial x_{1}^{k_{1}}\ldots\partial x_{n}^{k_{n}}}$ are
polynomially bounded on $x_{1},\ldots,x_{n}$, we can conclude that each
partial derivative of $s_{i}$ is polynomially bounded by some polynomial
$\tilde{p}\left(  \frac{y_{1}}{1-y_{0}},\ldots,\frac{y_{n}}{1-y_{0}}\right)
$, which implies that
\[
\lim_{\substack{y\rightarrow y_{NP}\\y\in\mathbb{S}^{n}-\{y_{NP}\}}
}\frac{\partial^{|k|}\varphi_{\ast}\left(  h\right)  (y)}{\partial t^{k_{0}
}\partial y_{1}^{k_{1}}\ldots\partial y_{n}^{k_{n}}}=0.
\]
Therefore we can extend $\varphi_{\ast}\left(  h\right)  (y)$ to a
$C^{\infty }$ vector field $\tilde{g}$ defined on the entire sphere
$\mathbb{S}^{n}$ if we assume that the value of $\tilde{g}$ and of
its partial derivatives is $0$ at the north pole
$y_{NP}=(1,0,\ldots,0)$ of $\mathbb{S}^{n}.$ We thus have defined a
$C^{\infty}$ vector field $\tilde{g}$ on the entire sphere
$\mathbb{S}^{n}$, and $\tilde{g}$ is Turing universal.
\end{proof}

With the help of the above theorem, we may hope we could just make
use of the vector field $g_{M}$ of Theorem
\ref{Th:two-dim-ODE-Turing} as the vector field $f$ in Theorem
\ref{Th:compact_simulation} to prove that there is a Turing
universal vector field in $\mathbb{S}^{2}$. However, there are
several problems with this approach as listed below: (1) the
approach requires that $g_{M}(t, x_1, x_2)$ and all of its partial
derivatives are bounded by a constant on $t$ and are polynomially
bounded on $x_{1}$ and $x_{2}$. A major problem in this respect is
that $g_{M}$ uses in its expression the function $\Psi$ defined in
Lemma \ref{Lemma:Psi} that does not necessarily have polynomially
bounded derivatives because the function $\arcsin$ is used in the
definition of $\Psi$ and the derivative of $\arcsin$ is not
polynomially bounded as its argument approaches $-1$ or $1$. (2) The
expression of $g_{M}$ relies on the expression of the function
$\bar{g}_{M}$
given by Theorem \ref{Th:Simulation}. Thus one must show that $\bar{g}
_{M}(t,x_{1},x_{2},x_{3},y_{1},y_{2},y_{3})$ is bounded polynomially
on $x_{1},x_{2},x_{3},y_{1},y_{2},y_{3}$. (3) The argument of the functions $\Omega_{k,i}
:\mathbb{R\rightarrow R}$ from Proposition \ref{Prop:one-to-many} is
provided as the initial condition of an ODE. It is not
straightforward to analyze the dependence of $\Omega_{k,i}$ and of
its derivatives on its argument.

In the following we present the solutions to the listed problems.
First we note that in Theorem \ref{Th:compact_simulation} the vector
field is no longer required to be analytic (it only has to be
$C^{\infty}$). Therefore we can substitute the function $\Psi$
defined in Lemma \ref{Lemma:Psi} by the function
$r:\mathbb{R\rightarrow R}$ defined in Lemma \ref{Lem:r}. We recall
that the function $r$ has the property that $r(x)=n$, whenever
$x\in\lbrack n-1/4,n+1/4]$, for all integers $n$. Thus if we take
$\Psi(x,y)=r(x)$, the properties stated for $\Psi$ in Lemma
\ref{Lemma:Psi} remain true. Therefore, we can replace $\Psi$ with
$r$ when defining the vector field $g_{M}$ of Theorem
\ref{Th:two-dim-ODE-Turing}. In this case, the properties stated in
Theorem \ref{Th:two-dim-ODE-Turing}\ remain true with $g_M$ being a
$C^{\infty}$ function rather than an analytic function. However, we
gain the advantage that $r$ and its derivatives are polynomially
bounded as we shall show now. Indeed, it follows from its definition
that $\left\vert r(x)\right\vert \leq\left\vert x\right\vert +1\leq
x^{2}+2$ (recall that $\left\vert x\right\vert \leq x^{2}+1$). For
the derivatives of $\theta$, we note that if $a(x)=e^{-1/x}$, then
it is readily seen by induction that for each $n\in\mathbb{N}_{0}$
there is a polynomial $P_{n}$ such that
$a^{(n)}(x)=P_{n}(1/x)e^{-1/x}$, with $a^{(0)}(x)=a(x)$ (and thus
$P_{0}=1$). This implies that $\lim_{x\rightarrow0^{+}}a^{(n)}(x)=0$
and $\lim_{x\rightarrow+\infty}a^{(n)}(x)=K_{n}$, where $K_{n}$ is
the constant term of $P_{n}$. Therefore, by definition of the limit,
there exists $b_{n},\varepsilon_{n}>0$ such that $\left\vert
a^{(n)}(x)\right\vert \leq1$ whenever $x\in(0,\varepsilon_{n}]$ and
$\left\vert a^{(n)}(x)\right\vert \leq K_{n}+1$ whenever $x\geq
b_{n}$. Now set $M_{n}=\max_{x\in \lbrack\varepsilon,M]}\left\vert
a^{(n)}(x)\right\vert \mathbb{\in R}$. Then $\left\vert
a^{(n)}(x)\right\vert \leq\max(1,K_{n}+1,M_{n})$ for all
$x\in(0,+\infty)$, which implies that $\theta$ in (\ref{Eq:theta})
as well as its derivatives are polynomially bounded on
$[0,+\infty)$. Consequently, the function $v$ from (\ref{Def:r}) as
well as the derivatives of $v$ are polynomially bounded following
Lemmas \ref{Lem:Arithm} and \ref{Lem:Comp} to be presented in a
moment. Lemma \ref{Lem:r} together with Lemmas \ref{Lem:Arithm} and
\ref{Lem:Comp} then imply that $r$ as well as its derivatives are
polynomially bounded.

Some notations are in order for the statements of the next two
lemmas. Given
multi-indexes $\alpha=(\alpha_{1},\ldots,\alpha_{n}),\beta=(\beta_{1}
,\ldots,\beta_{n})\in\mathbb{N}_{0}^{n}$ and $x=(x_{1},\ldots,x_{n}
)\in\mathbb{R}^{n}$, let
\begin{align*}
\left\vert \alpha\right\vert  &  =\alpha_{1}+\ldots+\alpha_{n}\\
\alpha+\beta &  =(\alpha_{1}+\beta_{1},\ldots,\alpha_{n}+\beta_{n})\\
\alpha!  &  =(\alpha_{1}!)\cdot(\alpha_{2}!)\cdot\ldots\cdot(\alpha_{n}!)\\
\beta &  \leq\alpha\text{ iff }\beta_{1}\leq\alpha_{1},\ldots,\beta_{n}
\leq\alpha_{n}\\
\binom{\alpha}{\beta}  &  =\binom{\alpha_{1}}{\beta_{1}}\ldots\binom
{\alpha_{n}}{\beta_{n}}=\frac{\alpha!}{\beta!(\alpha-\beta)!}\text{ for }
\beta\leq\alpha\\
D_{x}^{\alpha}  &  =\frac{\partial^{\left\vert \alpha\right\vert }}{\partial
x_{1}^{\alpha_{1}}\ldots\partial x_{n}^{\alpha_{n}}}\text{ \ \ (}D_{x}
^{0}\text{ is the identity operator)}\\
x^{\alpha}  &  =x_{1}^{\alpha_{1}}\ldots x_{n}^{\alpha_{n}}
\end{align*}

\begin{lemma}
\label{Lem:Arithm}Suppose that $f,g:\mathbb{R}^{n}\rightarrow\mathbb{R}$ are
$C^{\infty}$ functions which are bounded by a constant on the variables
$x_{1},\ldots,x_{i}$ and are bounded by a polynomial on the variables
$x_{i+1},\ldots,x_{n}$, as well as all their partial derivatives. Then:

\begin{enumerate}
\item The $C^{\infty}$ function $f\pm g:\mathbb{R}^{n}\rightarrow\mathbb{R}$
defined by $(f+g)(x)=f(x)+g(x)$ is bounded by a constant on the variables
$x_{1},\ldots,x_{i}$ and bounded by a polynomial on the variables
$x_{i+1},\ldots,x_{n}$, as well as all its partial derivatives.

\item The $C^{\infty}$ function $f\times g:\mathbb{R}^{n}\rightarrow
\mathbb{R}$ defined by $(f\times g)(x)=f(x)\cdot g(x)$ is bounded by a
constant on the variables $x_{1},\ldots,x_{i}$ and bounded by a polynomial on
the variables $x_{i+1},\ldots,x_{n}$, as well as all its partial derivatives.
\end{enumerate}
\end{lemma}

\begin{proof}
Let $\alpha=(\alpha_{1},\ldots,\alpha_{n})$ be a multi-index. For point 1, we
note that
\[
\frac{\partial^{\left\vert \alpha\right\vert }}{\partial x^{\alpha}}(f\pm
g)(x)=\frac{\partial^{\left\vert \alpha\right\vert }f(x)}{\partial x^{\alpha}
}\pm\frac{\partial^{\left\vert \alpha\right\vert }g(x)}{\partial x^{\alpha}}
\]
and thus the result follows immediately from the assumption.

For the product $f\times g$, the claim follows directly from the
general Leibniz rule for multivariate functions (see
e.g.~\cite[Proof of Lemma 2.6]{CS96}):
\[
\frac{\partial^{\left\vert \alpha\right\vert }}{\partial x^{\alpha}}(f\times
g)(x)=\sum_{0\leq\beta\leq\alpha}\binom{\alpha}{\beta}\frac{\partial
^{\left\vert \beta\right\vert }f(x)}{\partial x^{\beta}}\frac{\partial
^{\left\vert \alpha-\beta\right\vert }g(x)}{\partial x^{\alpha-\beta}}.
\]
\end{proof}

\begin{lemma}
\label{Lem:Comp}Suppose that $f:\mathbb{R}^{j}\rightarrow\mathbb{R}^{n}$ and
$g:\mathbb{R}^{k}\rightarrow\mathbb{R}^{j}$ are $C^{\infty}$ functions with
the following properties:

\begin{enumerate}
\item $f$ and its partial derivatives are polynomially bounded on its
arguments $z_{1},\ldots,z_{j}$;

\item $g$ and its partial derivatives are bounded by a constant on the
variables $x_{1},\ldots,x_{i}$ and bounded by a polynomial on the variables
$x_{i+1},\ldots,x_{k}$.
\end{enumerate}

Then $f\circ g$ is a $C^{\infty}$ function with the property that
$f\circ g$ as well as all its partial derivatives are bounded by a
constant on the variables $x_{1},\ldots,x_{i}$ and by a polynomial
on the variables $x_{i+1},\ldots,x_{k}$.
\end{lemma}

\begin{proof}
\label{Lem-Faa}To prove this theorem, we will use a multivariate
version of the Fa\`{a} di Bruno formula which allows us to compute
the higher order partial derivatives of the composition of
multivariate functions $f$ and $g$. Some notational matter is in
order first.  Let
$\alpha=(\alpha_{1},\ldots,\alpha_{n})\in\mathbb{N}_{0}^{n}$ be a
multi-index. Then following the approach of \cite{Ma09}, let us
assume that $\chi^{0}=1$ regardless of whether $\chi$ is a number or
a differential operator. We say
that a multi-index $\alpha$ can be decomposed into $s$ parts $p_{1}
,\ldots,p_{s}\in\mathbb{N}_{0}^{n}$ with multiplicities $m_{1},\ldots,m_{s}
\in\mathbb{N}_{0}^{d}$ if the decomposition $\alpha=\left\vert m_{1}
\right\vert p_{1}+\ldots+\left\vert m_{s}\right\vert p_{s}$ holds and all
parts are different. In this case the total multiplicity is defined as
$m=m_{1}+\ldots+m_{s}$. The list $(s,p,m)$ is called a $d$-decomposition, or
simply just a decomposition, of $\alpha$. Then, assuming that $z=f\circ g(x)$
and $y=g(x)$, we have
\[
\frac{\partial^{\left\vert \alpha\right\vert }}{\partial x^{\alpha}}f\circ
g(x)=\alpha!\sum_{(s,p,m)\in\mathcal{D}}\frac{\partial^{\left\vert
m\right\vert }}{\partial y^{m}}f(g(x))\prod_{k=1}^{s}\frac{1}{m_{k}!}\left(
\frac{1}{p_{k}!}\frac{\partial^{\left\vert p_{k}\right\vert }}{\partial
x^{p_{k}}}g(x)\right)  f\circ g(x)
\]
where $\mathcal{D}$ is the set of all decompositions of $\alpha$. From this
later formula and the the hypothesis on $f$, $g$ we conclude that $f\circ g$
is a $C^{\infty}$ function with the property that $f\circ g$ as well as all
its partial derivatives are bounded by a constant on the variables
$x_{1},\ldots,x_{i}$ and are bounded by a polynomial on the variables
$x_{i+1},\ldots,x_{k}$.
\end{proof}

Since the function $f_{M}$ of Theorem \ref{Th:Main} can be written
using only the following terms: variables, polynomial-time
computable constants, $+$, $-$, $\times$, $\sin$, $\cos$, $\arctan$,
we conclude from Lemmas \ref{Lem:Arithm} and \ref{Lem:Comp} that
$f_{M}$ as well as all its partial derivatives are bounded by a
polynomial. Furthermore, as the function $g_{M}$ from Theorem
\ref{Th:one-dim-Turing-map} is obtained using the functions
$\Upsilon_{3},f_{M},\Omega_{3,1},\Omega_{3,2},\Omega_{3,3},\sigma$,
again by Lemmas \ref{Lem:Arithm} and \ref{Lem:Comp} it suffices to
show that
$\Upsilon_{3},\Omega_{3,1},\Omega_{3,2},\Omega_{3,3},\sigma$ as well
as their partial derivatives are (or can be made) bounded by a
polynomial. The case of the function $\sigma$ in Lemma
\ref{Lemma:sigma} is immediate from Lemmas \ref{Lem:Arithm} and
\ref{Lem:Comp}. The case of $\Upsilon_{3}$ can be treated by
replacing $\Psi(x,y)$ in the expression of $\Upsilon_{3}$ given in
the proof of Proposition \ref{Prop:many-to-one} by the $C^{\infty}$
function $r$, as explained above. Then it follows immediately that
$\Upsilon_{3}$ maintains its properties given by Proposition
\ref{Prop:many-to-one}, except analiticity ($\Upsilon_{3}$ will only
be $C^{\infty}$) and, meanwhile, Lemmas \ref{Lem:Arithm} and
\ref{Lem:Comp} imply that $\Upsilon_{3}$ and its partial derivatives
are polynomially bounded.

The situation for $\Omega_{3,1},\Omega_{3,2},\Omega_{3,3}$ is more subtle, as
the arguments to these functions are passed as initial conditions of an ODE.
What we are going to do is to create new $C^{\infty}$ functions $\overline{\Omega}
_{3,1},\overline{\Omega}_{3,2},\overline{\Omega}_{3,3}$ such that
these new functions maintain the useful properties of
$\Omega_{3,1},\Omega_{3,2},\Omega_{3,3}$ on the one hand and, on the
other hand, the new functions together with their partial
derivatives are polynomially bounded. Once this is done, the old
functions $\Omega_{3,1},\Omega_{3,2},\Omega_{3,3}$  can then be
replaced by the new functions
$\overline{\Omega}_{3,1},\overline{\Omega}_{3,2},\overline{\Omega}_{3,3}$
in the expression of $g_{M}$ in the proof of Theorem
\ref{Th:one-dim-Turing-map}. The function $g_M$ as well as its
partial derivatives are now polynomially bounded on all variables
except the time $t$. But since the time variable only appears inside
the functions $\phi_1$ and $\phi_2$ defined by (\ref{SistemaCont})
in the format of $\sin (2\pi t)$ as defined in (\ref{Eq:Def_s}), and
$\sin (2\pi t)$ and its derivatives are obviously bounded by a
constant, it follows that the theorem below holds true.

\begin{theorem}
\label{Th:2Dcompact-extension}Let $\psi:\mathbb{N}\rightarrow\mathbb{N}$ be
the transition function of a Turing machine $M$, under the encoding described
in Section \ref{Sec:Turing-map}. Then there exist $0<\eta<2/5$ and a
$C^{\infty}$ function $g_{M}:\mathbb{R}^{3}\rightarrow\mathbb{R}^{2}$ such
that the ODE $z^{\prime}=g_{M}(t,z)$ simulates $M$ in the following sense: for
all $x_{0}\in\mathbb{N}$ which encodes a configuration according to the
encoding described above, if $\bar{x}_{0},\bar{y}_{0}\in\mathbb{R}$ satisfy
the conditions $\left\Vert \bar{x}_{0}-x_{0}\right\Vert \leq1/5$ and
$\left\Vert \bar{y}_{0}-x_{0}\right\Vert \leq1/5$, then the solution $z(t)$
of
\[
z^{\prime}=g_{M}(t,z),\qquad z(0)=(\bar{x}_{0},\bar{y}_{0})
\]
satisfies, for all $j\in\mathbb{N}_{0}$ and for all $t\in\lbrack j,j+1/2]$,
\[
\left\Vert z_{2}(t)-\psi^{\lbrack j]}(x_{0})\right\Vert \leq\eta,
\]
where $z(t)\equiv(z_{1}(t),z_{2}(t))\in\mathbb{R}^{2}$. Furthermore
$g_{M}(t,z_{1},z_{2})$ and its partial derivatives are polynomially bounded on
$z_{1}$ and $z_{2}$ and bounded by a constant on $t$.
\end{theorem}

\begin{theorem}
Let $M$ be a Turing machine. Then one can compute from $f$ a $C^{\infty}$
vector field $F$ defined on $\mathbb{S}^{2}$ which also simulates $M$.
\end{theorem}

\begin{proof}
Immediate from Theorems \ref{Th:compact_simulation} and
\ref{Th:2Dcompact-extension}.
\end{proof}

It remains to show that the new $C^{\infty}$ functions
$\overline{\Omega}_{3,1},\overline{\Omega}_{3,2},\overline{\Omega}_{3,3}$
can be constructed such that they retain the useful properties of $\Omega_{3,1},\Omega_{3,2}
,\Omega_{3,3}$, but these new functions and their partial
derivatives are polynomially bounded. We begin with a preliminary
lemma.

\begin{lemma}
\label{Lem:BoundedODEs}Let
$f:\mathbb{R}^{n+1}\rightarrow\mathbb{R}^{n}$ be a $C^{\infty}$
function such that $f$ and its partial derivatives are polynomially
bounded. Suppose that $x_{f}:\mathbb{R\rightarrow R}$ is the first
coordinate of a solution $x$ of the IVP
\[
x^{\prime}=f(t,x).
\]
If $x$ is polynomially bounded, then all the
derivatives of $x_{f}$ are polynomially bounded.
\end{lemma}

\begin{proof}
We show the result by induction on the order of the derivative by showing that
$x^{(k)}=f^{k}(t,x(t))$, where $f^{k}$ is a $C^{\infty}$ function which is
polynomially bounded, as well as all its partial derivatives. Then we will
conclude that $x^{(k)}$ is polynomially bounded as well as all its partial
derivatives by Lemma \ref{Lem:Comp}. The base case
\[
x^{\prime}(t)=f(t,x(t))
\]
is trivial. Let us now assume that $x^{(k)}=f^{k}(t,x(t))$, where $f^{k}$ is
polynomially bounded as well as all its partial derivatives. Then
\begin{align*}
x_{i}^{(k+1)}(t)  &  =\left(  f_{i}^{k}(t,x(t))\right)  ^{\prime}\\
&  =\frac{\partial f_{i}^{k}}{\partial t}(t,x(t))+\sum_{j=1}^{n}\frac{\partial
f_{i}^{k}}{\partial y_{j}}(t,x(t))\frac{dx_{j}}{dt}(t)\\
&  =\frac{\partial f_{i}^{k}}{\partial t}(t,x(t))+\sum_{j=1}^{n}\frac{\partial
f_{i}^{k}}{\partial y_{j}}(t,x(t))f_{j}^{k}(t,x(t))\\
&  =f_{i}^{k+1}(t,x(t)).
\end{align*}
By taking $f^{k+1}=(f_{1}^{k+1},\ldots,f_{n}^{k+1})$, we conclude that
$x^{(k+1)}=f^{k+1}(t,x(t))$ and by Lemmas \ref{Lem:Arithm} and \ref{Lem:Comp}
we conclude that $f^{k+1}$ is polynomially bounded as well as all its partial
derivatives, thus showing the result.
\end{proof}

To define new $C^{\infty}$ functions
$\overline{\Omega}_{3,1},\overline
{\Omega}_{3,2},\overline{\Omega}_{3,3}$ as in Proposition
\ref{Prop:one-to-many}, we recall that the key point of the proof of
Proposition \ref{Prop:one-to-many} was to consider the bijection
$I:\mathbb{N}^{2}\rightarrow\mathbb{N}$ given by (\ref{Eq:bijection-dim2-dim1}
) and to obtain real extensions of the components $J_{2,1}$ and $J_{2,2}$
which form the inverse function of $I$, i.e.~$I^{-1}(z)=(J_{2,1}
(z),J_{2,2}(z))$. Then the result would follow inductively when
obtaining extensions $\Omega_{k,i}$ from $J_{k,i}$ for $k>2$.
Subsequently, the only required modification is to obtain suitable
real extensions $\overline {\Omega}_{2,1},\overline{\Omega}_{2,2}$
of $J_{2,1},J_{2,2}$, respectively. We shall demand that if
$n\in\mathbb{N}_{0}$, then $\left\vert \overline{\Omega
}_{2,i}(z)-J_{2,i}(n)\right\vert \leq1/5$ whenever $\left\vert
z-n\right\vert \leq1/4$ for $i=1,2$, so that
$\overline{\Omega}_{2,i}$ has the same properties as of
$\Omega_{2,i}$ regarding Proposition \ref{Prop:one-to-many}, except
that $\overline{\Omega}_{2,i}$ is $C^{\infty}$ instead of analytic.
We also require that $\overline{\Omega}_{2,i}$ and its derivatives
are polynomially bounded. To obtain $\overline{\Omega}_{2,i}$, we
first construct a $C^{\infty}$ function $\tilde{\Omega}_{2,i}$ with
the property that $\left\vert
\tilde{\Omega}_{2,i}(z)-J_{2,i}(n)\right\vert \leq1/4$ whenever
$z\in\lbrack n,n+1/2]$ for $i=1,2$, and $\tilde{\Omega}_{2,i}$ as
well as its partial derivatives are polynomially bounded. By setting
$\overline{\Omega}_{2,i}(z)=\sigma
\circ\tilde{\Omega}_{2,i}(z+1/4)$, where $\sigma$ is given by Lemma
\ref{Lemma:sigma}, we conclude from the above and from Lemmas
\ref{Lem:Arithm} and \ref{Lem:Comp} that $\overline{\Omega}_{2,i}$
has the desired properties. Before defining $\tilde{\Omega}_{2,1}$
and $\tilde{\Omega}_{2,2}$, we note that $I$ is obtained by
dovetailing by enumerating the pairs in the diagonals below from the
(one element) diagonal starting on $(0,0)$ and then moving to the
next diagonal

\begin{center}
\begin{tikzcd}
    (0,0) & (0,1) \arrow[ld] & (0,2) \arrow[ld]& (0,3) \arrow[ld]& \ldots\\
    (1,0) & (1,1) \arrow[ld] & (1,2) \arrow[ld] & \ldots&  \\
    (2,0) & (2,1) \arrow[ld] & \ldots& & \\
    (3,0) & \ldots& & &
\end{tikzcd}
\end{center}

\noindent In other words, we have $I(0,0)=0,$ $I(0,1)=1$,
$I(1,0)=2$, $I(0,2)=3$, and so on. We note that the sum of the
coordinates in each diagonal is constant. From here we see that the
graphs for $J_{2,1}$ and $J_{2,2}$, provided in Figures
\ref{fig:J_2_1} and \ref{fig:J_2_2}, respectively, have certain
regularities which will be explored to obtain
$\tilde{\Omega}_{2,1}$ and $\tilde{\Omega}_{2,2}$.
\begin{figure}
    \begin{center}
        \includegraphics[width=12cm]{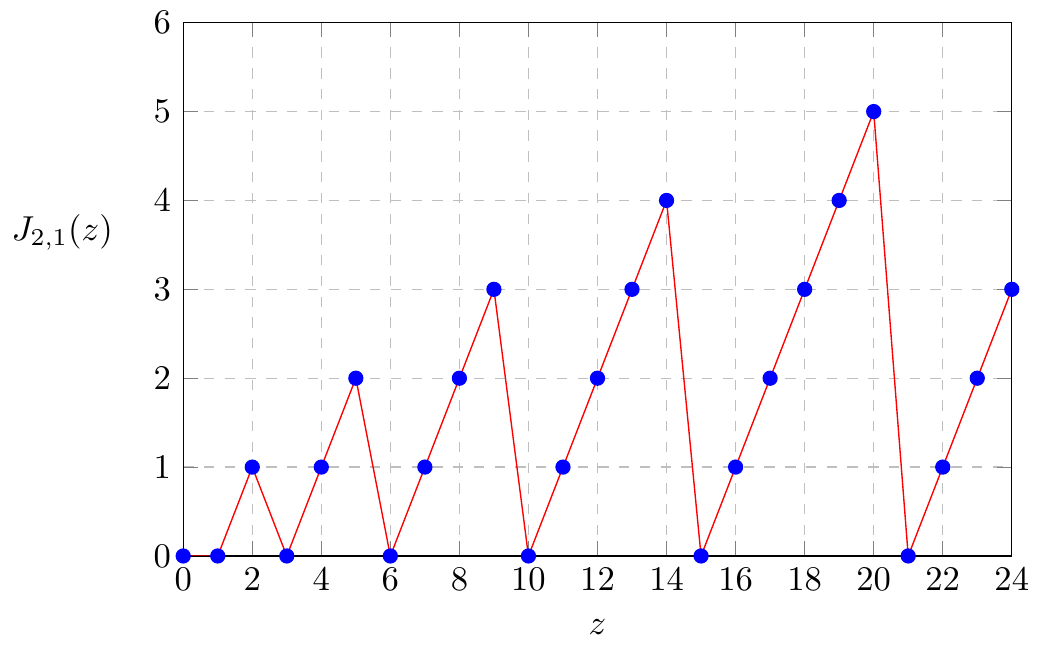}
        \caption{Graph of the function $J_{2,1}$.
Since the function is discrete, the image are only the blue points (the red line is given as a visualization helper).}
        \label{fig:J_2_1}
    \end{center}
\end{figure}
\begin{figure}
    \begin{center}
        \includegraphics[width=12cm]{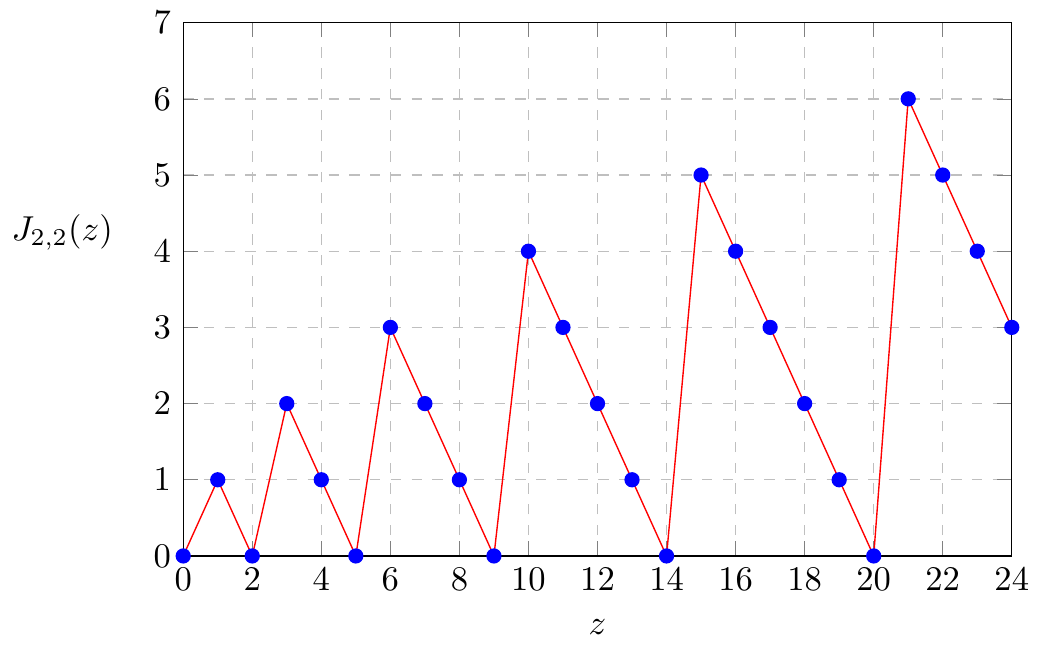}
        \caption{Graph of the function $J_{2,2}$.
Since the function is discrete, the image are only the blue points (the red line is given as a visualization helper).}
        \label{fig:J_2_2}
    \end{center}
\end{figure}
Let us start with the case of $\tilde{\Omega}_{2,1}$. We first analyze the
behavior of $J_{2,1}$. Let us suppose that the argument $z$ of $J_{2,1}(z)$
codes a pair $(0,n)$ at the start of the diagonal with sum $n$. Then
$J_{2,1}(z)=0$, $J_{2,1}(z+1)=1,$ \ldots, $J_{2,1}(z+n)=n$, $J_{2,1}
(z+n+1)=0$, $J_{2,1}(z+n+2)=1,$ \ldots\ Thus to simulate $J_{2,1}$ we need to
track the sum $s$ of the diagonal, and increase it by one when $J_{2,1}(z)$
reaches the value of $s$. On the next value $z+1$, we will have that
$J_{2,1}(z+1)$ will take the value $0$, and each time its argument $z$
increases by one, $J_{2,1}(z)$ also increases by one, until it reaches the
(new) value of the sum of the diagonal and the cycle repeats itself. We will
simulate this behavior with ODEs. Before showing how this can be done,
consider the auxiliary function $\xi:\mathbb{R\rightarrow R}$ defined as
\[
\xi^{\prime}(x)=\left\{
\begin{array}
[c]{ll}
0 & \text{if }x\leq1/4\\
c_{\xi}\theta(-(x-1/4)(x-3/4))\text{ \ \ } & \text{if }1/4<x<3/4\\
0 & \text{if }x\geq3/4
\end{array}
\right.
\]
where $\xi(x)=0$ and $c_{\xi}=\left(  \int_{1/4}^{3/4}\theta
(-(x-1/4)(x-3/4))dx\right)  ^{-1}$. It is not difficult to see that
$\theta(-(x-1/4)(x-3/4))>0$ when $1/4<x<3/4$. Hence we have that $\xi(x)=0$
whenever $x\leq1/4$, $0<\xi(x)<1$ when $1/4<x<3/4$ and $\xi(x)=1$ when
$x\geq3/4$. Furthermore $\xi$ is $C^{\infty}$ and $\xi$ and all its
derivatives are polynomially bounded due to Lemmas \ref{Lem:Arithm} and
\ref{Lem:Comp}.

Let us now present the ODE which will define $\tilde{\Omega}_{2,1}$
\begin{equation}
\left\{
\begin{array}
[c]{l}
x_{1}^{\prime}=\tilde{c}(\xi(r(s_{2})-r(x_{2}))(1+r(x_{2}))-x_{1})^{3}
\theta(\sin2\pi t)\\
x_{2}^{\prime}=\tilde{c}(r(x_{1})-x_{1})^{3}\theta(-\sin2\pi t)\\
s_{1}^{\prime}=\tilde{c}(r(s_{2})+\xi(r(x_{2})+1-r(s_{2}))-s_{1})^{3}
\theta(\sin2\pi t)\\
s_{2}^{\prime}=\tilde{c}(r(s_{1})-s_{1})^{3}\theta(-\sin2\pi t)
\end{array}
\right.  \label{Eq:J_2_1}
\end{equation}
with $x_{1}(0)=x_{2}(0)=s_{1}(0)=s_{2}(0)$ and $\tilde{c}=206$. The behavior
of the ODE (\ref{Eq:J_2_1})\ is similar to the one of (\ref{iteration}). The
variable updates are done on alternating time intervals. The variable $x_{2}$
stores the value of the function $J_{2,1}$ on time intervals with the format
$[k,k+1/2]$, i.e.~we will have $\left\vert z_{2}(t)-J_{2,1}\right\vert
\leq1/4$ whenever $t\in\lbrack k,k+1/2]$, with $k\in\mathbb{N}_{0}$. The
variable $s_{2}$ will give the current sum of the diagonal on time intervals
with the format $[k,k+1/2]$. We first update the variables $x_{1}$ and $s_{1}$
on time intervals $[k,k+1/2]$ to be able to use the \textquotedblleft
memorized\textquotedblright\ values of $x_{2}$ and $s_{2}$ when updating
$x_{1}$ and $s_{1}$. We note that $x_{1}$ must be increased by one unit from
its previous value (stored on $x_{2}$) until it reaches the value of the sum
of the diagonal, which is stored in $s_{2}$. On that moment we will have
$\xi(r(s_{2})-r(x_{2}))=0$ on the equation for $x_{1}^{\prime}$ (if the value
of $x_{2}$ is less than the value of the sum of the diagonal stored in $s_{2}
$, then $\xi(r(s_{2})-r(x_{2}))=1$ and $x_{1}$ is incremented by one) and
$x_{1}$ will be reset to the value $0$ starting the cycle again. The analysis
for $s_{1}$ is similar: its value will be essentially constant as long as
$\xi(r(x_{2})+1-r(s_{2}))=0$, which happens when $r(s_{2})-1\geq r(x_{2})$.
When $r(x_{2})=$ $r(s_{2})$, we will have $\xi(r(x_{2})+1-r(s_{2}))=1$ and
$s_{1}$ will be incremented by one from its previous value. We can see that
the ODE (\ref{Eq:J_2_1}) behaves as desired and that we can take
$\tilde{\Omega}_{2,1}(t)=x_{2}(t)$. Since the right-hand sides of
(\ref{Eq:J_2_1}) are polynomially bounded as well as their derivatives and
$(x_{1},x_{2},s_{1},s_{2})$ is also polynomially bounded, implying by Lemma
\ref{Lem:BoundedODEs} that $\tilde{\Omega}_{2,1}$ and all its derivatives are
polynomially bounded.

The case for $J_{2,2}$ is similar. Let us suppose that the argument $z$ of
$J_{2,2}(z)$ codes a pair $(0,n)$ at the start of the diagonal with sum $n$.
Then $J_{2,2}(z)=n$, $J_{2,2}(z+1)=n-1,$ \ldots, $J_{2,2}(z+n)=0$,
$J_{2,2}(z+n+1)=n+1$, $J_{2,2}(z+n+2)=n,$ \ldots\ Thus to simulate $J_{2,2}$
we need again to track the sum $s$ of the diagonal, but now we need to
increase it by one when $J_{2,1}(z)$ reaches the value $0$. On the next value
$z+1$, we will have that $J_{2,1}(z+1)$ will take the value $s+1$, and each
time its argument $z$ increases by one, $J_{2,1}(z)$ decreases by one, until
it reaches $0$ and the cycle repeats itself. This behavior can be simulated in
a similar way to (\ref{Eq:J_2_1}) by the following ODE
\begin{equation}
\left\{
\begin{array}
[c]{l}
x_{1}^{\prime}=\tilde{c}(\xi(x_{2})(r(x_{2})-1)+\xi(1-x_{2})(1+r(s_{2}
))-x_{1})^{3}\theta(\sin2\pi t)\\
x_{2}^{\prime}=\tilde{c}(r(x_{1})-x_{1})^{3}\theta(-\sin2\pi t)\\
s_{1}^{\prime}=\tilde{c}(\xi(x_{2})(r(s_{2})+\xi(1-x_{2})(r(s_{2}
)+1)-s_{1})^{3}\theta(\sin2\pi t)\\
s_{2}^{\prime}=\tilde{c}(r(s_{1})-s_{1})^{3}\theta(-\sin2\pi t)
\end{array}
\right.  \label{Eq:J_2_2}
\end{equation}

We can do an analysis to (\ref{Eq:J_2_2}) similar to the one of
(\ref{Eq:J_2_1}) to conclude that we can take $\tilde{\Omega}_{2,2}
(t)=x_{2}(t)$ on (\ref{Eq:J_2_2}). This concludes the proof of Theorem
\ref{Th:2Dcompact-extension}.

\section{Can one-dimensional ODEs simulate Turing
machines?\label{Sec:ODE-one-dimension}}

As we have seen in the previous section, analytic two-dimensional ODEs can
robustly simulate Turing machines. But what about one-dimensional ODEs? In
this section we show that no one-dimensional autonomous ODE can simulate a
universal Turing machine under some reasonable conditions.

First let us give a more precise meaning to the notion of an ODE simulating a
Turing machine. Let $M$ be a Turing machine. Since ODEs are defined on
$\mathbb{R}^{k}$, to simulate the Turing machine $M$ with an ODE we first need
to encode a configuration of $M$ as a point of $\mathbb{R}^{k}$. However,
since the coding of a configuration might not be unique, as it happens in the
previous sections, we map each configuration to a \emph{set of possible
encodings of that configuration}. Hence we have to consider a map $\chi$ which
maps configurations of $M$ into non-empty subsets of $\mathbb{R}^{k}$. Then
given a configuration $c_{M}$, any point of $\chi(c_{M})$ is assumed to
represent the configuration $c_{M}$. In this manner we can consider the case
when $c_{M}$ is represented by a single point in $\mathbb{R}^{k}$ (when
$\chi(c_{M})$ is a singleton) or when $c_{M}$ is represented by several points
of $\mathbb{R}^{k}$. For example, in Theorem \ref{Th:one-dim-Turing-map} we
have assumed that any point in a $1/5$-vicinity of $I_{3}(y_{1},y_{2},q)$,
where $y_{1}$ and $y_{2}$ are given by (\ref{Eq:ConfigEncoding}) represents
the configuration $c_{M}$ which is encoded by $I_{3}(y_{1},y_{2},q)$, i.e.
\[
\chi(c_{M})=\{x\in\mathbb{R}:\left\vert x-I_{3}(y_{1},y_{2},q)\right\vert
\leq1/5\}.
\]
Note that it makes sense to assume that if $c_{M}$ and $c_{M}^{\prime}$ are
distinct configurations, then $\chi(c_{M})\cap\chi(c_{M}^{\prime}
)=\varnothing$. However, this assumption may be too weak, since even
if $\chi(c_{M})\cap\chi(c_{M}^{\prime})=\varnothing$ nothing
prevents e.g.~that $\chi(c_{M})$ and $\chi(c_{M}^{\prime})$ are
fractal (e.g.~Cantor-like) sets which are intermingled and thus very
hard to separate in practice. To avoid such undesirable instances we
impose a natural separation-condition on $\chi$ so that
$\chi(c_{M})$ and $\chi(c_{M}^{\prime})$ are separated by disjoint
open subsets of $\mathbb{R}^k$. More precisely, let
$\{c_{i}\}_{i\in\mathbb{N}}$ denote all configurations of a given
Turing machine $M$ (recall that a Turing machine has at most
countably many configurations).
Then we assume that there are two computable maps
$a:\mathbb{N\rightarrow Q}^{k}$, $r:\mathbb{N\rightarrow Q}$ such
that for all $i\in\mathbb{N}$,
$\chi(c_{i})\subseteq\overline{B(a(i),r(i))}=\{x\in\mathbb{R}^{k}:\left\Vert
x-a(i)\right\Vert \leq r(i)\}$ and, moreover, if $i\neq j$ then
$\overline {B(a(i),r(i))}\cap\overline{B(a(j),r(j))}=\varnothing$.

Now we say that the
ODE
\begin{equation}
y^{\prime}=f(y), \label{Eq:simulation}
\end{equation}
where $f:\mathbb{R}^{k}\rightarrow\mathbb{R}^{k}$, simulates a
Turing machine with the coding $\chi$ if given an arbitrary
configuration $c_{0}$ of $M$ and some point $y_{0}\in\chi(c_{0})$
one has that the solution $y$ to (\ref{Eq:simulation}) with initial
condition $y(0)=y_{0}$ satisfies $y(n)\in\chi(\psi^{\lbrack
n]}(c_{0}))$ for all $n\in\mathbb{N}$, where $\psi$ is the
transition function of $M$. In the following, we show that no
one-dimensional ODE can simulate a universal Turing machine under
the separation-condition.

\begin{theorem}
Let $M$ be a universal Turing machine. Then no ODE $y^{\prime
}=f(y)$ can simulate $M$ in the sense explained above, where
$f:\mathbb{R\rightarrow R}$ is a computable function with only
isolated zeros.
\end{theorem}

\begin{proof}
Let $M$ be a universal Turing machine. We may assume that it has
only one halting state and  it cleans its tape immediately before it
halts. This implies that $M$ has only one halting configuration
$c_{k}$, $k\in\mathbb{N}$; moreover, the problem of deciding whether
$M$ halts on input $w$, $w\in\Sigma^{\ast }=\{0,1\}^{\ast}$, is
undecidable.

Let us assume, by contradiction, that there is an ODE
(\ref{Eq:simulation}) which simulates $M$, where
$f:\mathbb{R}\rightarrow\mathbb{R}$ is a computable function. Then
$f$ must admit a zero in $B_{k}=\overline
{B(a(k),r(k))}=[a_{k}-r_{k},a_{k}+r_{k}]$ where $a_{k}=a(k)$ and
$r_{k}=r(k)$. Assume otherwise that this is not the case. Then since
computable functions are continuous, it must be either $f(x)<0$ for
all $x\in B_{k}$ or $f(x)>0$ for all $x\in B_{k}$. Moreover, since
$B_{k}$ is compact, it follows that $\min_{x\in B_{k}}\left\vert
f(x)\right\vert =\delta>0$. As a result, any solution starting on
$B_{k}$ must leave it in time $\leq2r_{k}/\delta$ and never return
to $B_{k}$ afterwards (note that a solution of (\ref{Eq:simulation})
is a continuous function which must move continuously along the real
line). But this is impossible because $\psi^{\lbrack
n]}(c_{k})=c_{k}$ for all $n\in\mathbb{N}$ and (\ref{Eq:simulation})
simulates $M.$ Hence $B_{k}$ must contain at least one zero $x_{k}$,
which is computable because $f$ is computable and the zeros of $f$
are isolated (it is well-known that isolated zeros of computable
functions are computable. See e.g.~\cite[Theorem 7.8]{BHW08}).

Let $w$ be some input with the property that $M$ halts on $w$, and
suppose that the initial configuration associated to $w$ is
$c_{i_{w}}$. Then $c_{i_{w}}\neq c_{k}$ (note that in an universal
Turing machine the initial state cannot be an halting state). Hence,
if $y_{0}\in\chi(c_{i_{w}})$, then $y_{0}\notin B_{k}$. Let us
assume without loss of generality that $y_{0}<a_{k}-r_{k}$. We note
that the solution of the IVP (\ref{Eq:simulation}) with $y(0)=y_{0}$
must reach $B_{k}$ because $M$ halts on input $w$. This implies that
$f(x)>0$ for all $x\in\lbrack y_{0},a_{k}-r_{k})$. There are two
cases to to be considered:

\begin{enumerate}
\item $f(x)>0$ for all $x\in(-\infty,a_{k}-r_{k})$. In this case, let
$I=(-\infty,a_{k}+r_{k}]$.

\item $f(\bar{x})=0$ for some $\bar{x}\in(-\infty,a_{k}-r_{k})$. In this
case,
we must have $\bar{x}<y_{0}$, which implies that  $d=\min\{\left\vert x-y_{0}
\right\vert :x\in\lbrack\bar{x},y_{0}]$ and $f(x)=0\}>0$. By
continuity of $f$ there is some $\tilde{x}\in\lbrack\bar{x},y_{0}]$
such that $f(\tilde {x})=0$ and $\left\vert
\tilde{x}-y_{0}\right\vert =d$. In this case, let
$I=(\tilde{x},a_{k}+r_{k}]$ (note that $\tilde{x}$ is computable
because it is an isolated zero of a computable function).
\end{enumerate}

If there is a word $w$ such that $M$ halts on input $w$ with the
property that there is some $z\in\chi(c_{i_{w}})$ satisfying
$a_{k}<z$, we repeat the above procedure on the half line
$[a_{k},+\infty)$ obtaining $I:=I\cup\lbrack a_{k}+r_{k},+\infty)$
provided that $f(x)<0$ for all $x\in(a_{k}+r_{k},+\infty)$ or
$I:=I\cup\lbrack a_{k}+r_{k},\tilde{x})$ provided that $f(x)=0$ for
some $x\in(a_{k}+r_{k},+\infty)$, where $\tilde{x}$ is obtained
similarly as in the previous case.

From the arguments above, we conclude that $M$ halts on word $w$ iff
$\chi(c_{i_{w}})\subseteq I$. We will use this fact to show that the
halting problem is decidable, a contradiction. Let us assume,
without loss of generality that $I=(\tilde{x}_{1},\tilde{x}_{2})$
(the cases where one or more extremities of $I$ are unbounded is
dealt with similarly). Suppose first that
for all $i\in\mathbb{N}$, $\{\tilde{x}_{1},\tilde{x}_{2}\}\cap B_{i}
=\varnothing$. Then to decide whether $M$ halts on input $w$ proceed as
follows. Given the initial configuration $c_{i_{w}}$ associated to input $w$,
test whether $a_{i_{w}}\in I$ by testing whether $\tilde{x}_{1}<a_{i_{w}
}<\tilde{x}_{2}$. Notice that, since $\{\tilde{x}_{1},\tilde{x}_{2}\}\cap
B_{i_{w}}=\varnothing$, this test can be done in finite time. If the test
suceeds, then accept $w$ otherwise reject it.

Let us now suppose that $\tilde{x}_{1}\in B_{j_{1}}$ and $\tilde{x}_{2}\in
B_{j_{2}}$ (the cases where: (i) $\tilde{x}_{1}\in B_{j_{1}}$, and for all
$i\in\mathbb{N}$ $\tilde{x}_{2}\notin B_{i}$ or (ii) $\tilde{x}_{2}\in
B_{j_{2}\text{,}}$ and for all $i\in\mathbb{N}$ $\tilde{x}_{1}\notin B_{i}$
could be treated similarly). We can then \textquotedblleft
wire\textquotedblright\ into the above algorithm for deciding the halting
problem the correct answers for the two special cases when $i_{w}=j_{1}$ or
$i_{w}=j_{2}$. More specifically, given the initial configuration $c_{i_{w}}$
associated to input $w$, test if $i_{w}=j_{1}$. If the test $i_{w}=j_{1}$
suceeds, then accept (reject) if $M$ (does not, respectively) halts starting
on configuration $c_{j_{1}}$. Otherwise, test if $i_{w}=j_{2}$. If the test
$i_{w}=j_{2}$ suceeds, then accept (reject) if $M$ (does not, respectively)
halts starting on configuration $c_{j_{2}}$. If both tests fail, test whether
$a_{i_{w}}\in I$ by testing whether $\tilde{x}_{1}<a_{i_{w}}<\tilde{x}_{2}$.
Notice that in this case it must be $\{\tilde{x}_{1},\tilde{x}_{2}\}\cap
B_{i_{w}}=\varnothing$, and thus this test can be done in finite time. If the
test suceeds, then accept $w$ otherwise reject it.

In other words, if the ODE (\ref{Eq:simulation}) simulates $M$, then we can
decide the halting problem, a contradiction.
\end{proof}\medskip

\textbf{Acknowledgments.} D.~Gra\c{c}a wishes to thank Marco Mackaaij and Nenad Manojlovi\'{c} for helpful discussions.

\bibliographystyle{alpha}
\bibliography{ContComp}

\end{document}